\documentclass[11pt]{article}

\usepackage{psfrag}
\usepackage{amsmath, amssymb, amsthm}
\usepackage{mathptmx}
\usepackage{graphicx}
\usepackage{hyperref}
\usepackage{color}

\usepackage{setspace}
\makeatletter
\newsavebox\myboxA
\newsavebox\myboxB
\newlength\mylenA
\newcommand*\xoverline[2][0.75]{%
    \sbox{\myboxA}{$\m@th#2$}%
    \setbox\myboxB\null% Phantom box
    \ht\myboxB=\ht\myboxA%
    \dp\myboxB=\dp\myboxA%
    \wd\myboxB=#1\wd\myboxA% Scale phantom
    \sbox\myboxB{$\m@th\overline{\copy\myboxB}$}%  Overlined phantom
    \setlength\mylenA{\the\wd\myboxA}%   calc width diff
    \addtolength\mylenA{-\the\wd\myboxB}%
    \ifdim\wd\myboxB<\wd\myboxA%
       \rlap{\hskip 0.5\mylenA\usebox\myboxB}{\usebox\myboxA}%
    \else
        \hskip -0.5\mylenA\rlap{\usebox\myboxA}{\hskip 0.5\mylenA\usebox\myboxB}%
    \fi}
\makeatother

\newcommand{\arXiv}[1]{\href{https://arxiv.org/abs/#1}{arXiv:#1}}

\newcommand{\FF}{{\cal F}}
\newcommand{\DD}{{\cal D}}
\newcommand{\PP}{{\cal P}}

\newcommand{\eps}{{\varepsilon}}
\newcommand{\ch}{{\mbox{\rm ch}}}
\newcommand{\myth}{{\mbox{\rm th}}}

\newcommand{\e}{\mathbb{E}}

\newcommand{\p}{\mathbb{P}}

\newcommand{\Reals}{\mathbb{R}}

\newcommand{\la}{\langle}
\newcommand{\ra}{\rangle}

\newcommand{\ME}{\mbox{\it ME}}
\newcommand{\MCE}{\mbox{\it MCE}}

\newcommand{\bx}{\mathbf{x}}

\newcommand{\ignore}[1]{\relax}
\newcommand{\Distr}{\stackrel{d}{=}}

\newcommand{\lm}{\lambda}
\newcommand{\G}{\mathbf{G}}

\newcommand{\Gr}{\mathcal{G}}
\newcommand{\sg}{\sigma}
\newcommand{\pos}{\mathrm{Poisson}}

\newcommand{\pr}[1]{\mathbb{P}\left [ #1 \right]}
\newcommand{\E}[1]{\mathbb{E}\left [ #1 \right ]}

\newtheorem{theorem}{Theorem}

\newtheorem{lemma}{Lemma}[section]
\newtheorem{proposition}{Proposition}

\numberwithin{equation}{section}

\author{Wei-Kuo Chen\thanks{School of Mathematics, University of Minnesota. Email: wkchen@umn.edu}
		\and
		David Gamarnik\thanks{Sloan School of Management, MIT. Email: gamarnik@mit.edu}
		\and
		Dmitry Panchenko\thanks{Department of Mathematics, University of Toronto. Email: panchenk@math.toronto.edu}
		\and
		Mustazee Rahman\thanks{Department of Mathematics, MIT. Email: mustazee@mit.edu}
		}

\title{
Suboptimality of local algorithms for a class of max-cut problems
}
\date{}

\begin{document}

\maketitle

\begin{abstract}
 We show that in random $K$-uniform hypergraphs of constant average degree, for even $K \geq 4$, local algorithms 
 defined as factors of i.i.d. can not find nearly maximal cuts when the average degree is sufficiently large. These algorithms have been used frequently to obtain lower bounds for the max-cut problem on random graphs, but it was not known whether they could be successful in finding nearly maximal cuts. This result follows from the fact that the overlap of any two nearly maximal cuts in such hypergraphs does not take values in a certain non-trivial interval -- a phenomenon referred to as the overlap gap property -- which is proved by comparing diluted models with large average degree with appropriate fully connected spin glass models, and showing the overlap gap property in the latter setting.
\end{abstract}
\vspace{0.5cm}

\section{Introduction}\label{sec:intro}
This paper considers the problem of {algorithmically} finding nearly optimal spin configurations in the diluted $K$-spin model.
{We specifically focus on local algorithms defined as factors of i.i.d., the formal definition of which is provided in Section~\ref{sec:localg}}. The diluted $K$-spin model is also known as the max-cut problem for $K$-uniform Erd\H{o}s-R\'{e}nyi hypergraphs of constant average degree, and also as the random $K$-XORSAT model. The problem is only interesting for even $K$ and we prove that, for even $K \geq 4$, local algorithms fail to find the nearly optimal spin configurations (maximal cuts) once the average degree is large enough.

The proof is based on finding a structural constraint for the overlap of any two nearly optimal spin configurations -- the \emph{overlap gap property} -- that goes against certain properties of local algorithms. For $K=2$, the overlap gap property is not expected to hold, which is why this case is excluded. The structural constraint is derived from recent results on the mean field $K$-spin spin glass models, in particular, the Parisi formula and the Guerra-Talagrand replica symmetry breaking bound at zero temperature. We begin with a discussion of the model and the notion of algorithms that we use.

\paragraph{The $K$-spin model}
The set of $\pm 1$ spin configurations on $N$ vertices will be denoted by
$$\Sigma^N = \{-1,1\}^N.$$
Given any $K$-uniform hypergraph $G = ([N], E(G))$ on $N$ vertices, a spin configuration $\sg \in \Sigma^N$ and an 
edge $e = (v_1, \ldots, v_K) \in E(G)$, let us introduce the notation
\begin{equation} \label{eqn:theta} 
\theta(\sg,e) = - \, \sg_{v_1} \cdots \sg_{v_K}\,.
\end{equation}
The $K$-spin Hamiltonian of $G$ is the function $H_G: \Sigma^N \to \mathbb{R}$ defined by
\begin{equation} \label{eqn:Hgraph}
H_G(\sg) = \sum_{e \in E(G) }\, \theta(\sg,e).
\end{equation}
The \emph{cut density} of $\sg$ is $H_G(\sg)/N$. We consider directed hypergraphs and allow the same edge to appear multiple times, as well as a vertex to appear in an edge multiple times. In the case of the sparse Erd\H{o}s-R\'{e}nyi hypergraphs, which will be our main concern, this happens only for an order one number of edges and vertices with high probability as the graph size tends to infinity.

The objective of the $K$-spin model for a given graph $G$ is to maximize the cut density over all
spin configurations and to find the maximizers or near maximizers. When edges consist of distinct vertices, this objective is trivial if $K$ is odd with the maximum being achieved by the all $(-1)$-valued spin configuration. Therefore, throughout the paper, we will assume that the parameter $K$ is an even integer.

\paragraph{The diluted $K$-spin model}

The $K$-uniform Erd\H{o}s-R\'{e}nyi hypergraph on $N$ vertices with connectivity $\lm$, denoted  as $\mathrm{ER}(K, \lm, N)$, is sampled as follows. First, sample a random number $|E| \sim \pos(\lm N)$. Then for each $1 \leq i \leq |E|$, sample an \emph{edge} $e_i \in [N]^K$ uniformly at random, where $[N] = \{1,\ldots, N\}$. The graph $\G_N \sim \mathrm{ER}(K, \lm, N)$ has vertex set $V(\G_N) = [N]$ and edge set $E(\G_N) = \{e_1, \ldots, e_{|E|}\}$.

Let us consider the $K$-spin Hamiltonian for the graph $\mathrm{ER}(K,\lm,N)$:
\begin{equation} \label{eqn:HgraphTwo}
H_{\lm,N}(\sg) \triangleq H_{\G_N}(\sg) = \sum_{e \in E(\G)}\, \theta(\sg,e).
\end{equation}
This is the Hamiltonian of the so called diluted $K$-spin model.
Using Azuma's inequality it is easy to show that the maximum of $H_{\lm,N}(\sg)/N$ over all $\sg \in \Sigma^N$ is concentrated around its expectation as $N$ increases.
Moreover, the limit 
\begin{align}
M(K,\lm)\triangleq \lim_{N\to\infty} 
\e \max_{\sg \in \Sigma^N}\, \frac{H_{\lm,N}(\sg)}{N}
\end{align}
exists. This is proved for $K=2$ in \cite{BGT}. For even $K\geq 2$, the existence of this limit follows from the same argument as in \cite{DeSanctis}, which proves (for a different model) the existence of the thermodynamic limit of the free energy at positive temperature as a consequence of subadditivity -- obtained using an analogue of the Guerra-Toninelli interpolation \cite{GT} -- which, of course, implies the existence of the limit at zero temperature. The exact computation of $M(K,\lm)$ is however challenging, and, as we will discuss below, the appropriate large $\lm$ limit of $M(K,\lm)$ recovers the famous Parisi formula, which expresses the ground state energy of the $K$-spin Sherrington-Kirkpatrick model.

\paragraph{Local algorithms}

There is considerable interest in devising algorithms that generate spin configurations with provably high cut density for general diluted models.
Hatami et.~al.~\cite{HLS} conjectured that nearly optimal solutions to many such problems can be devised through local decision rules in the sense
of factor of i.i.d.~processes, as we explain below. Much effort has been made to address this question for problems such as perfect matchings
(dimer covers), independent sets (hardcore model), max-cut (spin model), graph colouring (Potts model), etc.
See \cite{BV,EL,FM,HV,HW,Lyons,LN,R} and references therein.

However, this conjecture was disproved in~\cite{GS} for the independent set model by using an approach based on the overlap gap property for
independent sets, similar to the approach used in this paper. The tightest possible result for the same problem
was obtained in \cite{RV} for half-optimal independent sets by looking at overlaps of many independents sets and establishing a multidimensional variant of the overlap gap property. The tightness of the result stems from the fact that less than half-optimal independent sets can be constructed by local
algorithms~\cite{HW}. 
Sub-optimality of a sequential version of a local algorithm was established for the random NAE-K-SAT model in~\cite{GS2}, also by establishing
the overlap gap property for multi-overlaps. The conjecture has remained open for the diluted $K$-spin model, and, in this paper, we prove that for $K \geq 4$ and all large enough $\lm$, local algorithms defined as factors of i.i.d.~can not find spin configurations whose cut density on $\mathrm{ER}(K, \lm, N)$ is arbitrarily close to $M(K,\lm)$.

Let us explain what we mean by local algorithms, specifically, algorithms defined as factors of i.i.d.
A formal definition is provided in the next section.
A \emph{local algorithm} takes as input a graph $G$ and decorates the
vertices of $G$ with i.i.d.~random variables $X(v)$ for $v \in V(G)$. The algorithm then generates a spin configuration $\sg$ by deciding, for each vertex $v$, the value $\sg_v$ as a function ({factor}) of the randomly decorated $r$-neighbourhood of $v$ for some radius $r$ that is independent of the size of $G$ (hence the term ``local"). The decision rule is the same for every vertex and is applied in parallel to all the vertices. The output of the algorithm is then a stochastic process over $G$.

Local algorithms are essentially randomized, distributed algorithms on graphs. They
originate from ergodic theory where they are called factor of i.i.d.~processes. Examples include continuous time
Glauber dynamics run for a bounded time interval \cite{BS2}, as well as the Belief Propagation and Survey Propagation based algorithms
run for a bounded number of iterations. Another well-known example is the R\"{o}dl nibble method for the packing problem.

Interest in local algorithms arose partly because any stationary stochastic process over the Cayley graph of a
finite generated amenable group, such as $\mathbb{Z}^d$, can be approximated in distribution by local algorithms
over the graph; see \cite{Lyons} for a discussion and references. On the other hand, there is a developing theory for
processes that can be modeled by local algorithms over non-amenable graphs and expander graph sequences,
such as random regular graphs or sparse Erd\H{o}s-R\'{e}nyi graphs with large enough connectivity (which behave
like expanders in many ways due to the giant component). In this setting, local algorithms can generate various graph
structures as previously indicated and those constructions provide the best lower bounds to date for many
optimization problems over sparse random graphs. Local algorithms can also approximate (in distribution) the local statistics
of eigenvectors of random regular graphs \cite{BS, HV} and certain stationary Gaussian processes over transitive graphs \cite{BV}.

\paragraph{Main result}
Our main result about sub-optimality of local algorithms is stated informally in Theorem \ref{thm:lafinformal} below
and then precisely in Theorem \ref{thm:laf}, after appropriate terminology is introduced in Section \ref{sec:localg}.

The \emph{mean field}, or fully connected, $K$-spin Hamiltonian is the random function $H_N: \Sigma^N \to \Reals$ given by
\begin{equation} \label{eqn:HSK}
H_N(\sg) = \frac{1}{N^{(K-1)/2}}\, \sum_{(i_1,\ldots, i_K) \in [N]^K}
g_{i_1, \ldots, i_K} \sg_{i_1} \cdots \sg_{i_K},
\end{equation}
where the $g_{i_1, \ldots, i_K}$ are i.i.d.~standard Gaussian random variables. Guerra and Toninelli proved in \cite{GT2} that
\begin{equation} \label{eqn:freeenergy}
\lim_{N \to \infty}
\e \max_{\sg \in \Sigma^N} \frac{H_N(\sg)}{N} \triangleq P(K)
\end{equation}
exists, and the precise description of this limit, given by the so called Parisi formula at zero temperature, will be discussed below and used crucially in the proof of the key results. Using another interpolation of Guerra and Toninelli from \cite{GT} that compares diluted models with large connectivity and fully connected models, Dembo et.~al.~\cite{DMS} (for $K=2$) and Sen \cite{Sen} showed that
\begin{equation} \label{eqn:DMS}
M(K,\lm) = P(K) \lm^{1/2} + O(\lm^{1/3})\;\; \text{as}\; \lm \to \infty.
\end{equation}
In other words, the leading contribution to $M(K,\lm)$ for large $\lambda$ is $P(K) \lm^{1/2}$. We prove the following
hardness of approximation result for $M(K,\lm)$, which will be made formal in Theorem \ref{thm:laf} below after necessary definitions.
\begin{theorem}[Informal] \label{thm:lafinformal}
	For any even $K \geq 4$, there exists $\eta > 0$ such that the following holds for all large enough $\lm$. Suppose $A$ is a local algorithm 
constructed as a factor of i.i.d.~that generates spin configurations on $K$-uniform hypergraphs. Let $\sg_{A} \in \Sigma^N$ be its output on the graph $\mathrm{ER}(K,\lm,N)$. Then, for all large enough $N$,
	$$ 
	\e \frac{H_{\lm,N}(\sg_A)}{N} \leq (P(K) - \eta) \lm^{1/2}.
	$$
\end{theorem}

\paragraph{A summary of the proof}
We informally describe the key steps in establishing this result. For $\eta>0$,
call a spin configuration $\sg \in \Sigma^N$  to be $\eta$-optimal if the cut density
$H_{\lm,N}(\sg)/N$ is at least $(P(K) - \eta) \lm^{1/2}$.  
A key ingredient in the proof of the main result is the \emph{overlap gap property}.
It is the existence, when $K \geq 4$, of an $\eta_1 > 0$ and a positive-length interval $[a,b] \subset [0,1]$
such that the absolute value of  the overlap, defined in \eqref{eqn:overlap}, of \emph{any two}
$\eta_1$-optimal spin configurations does not take values in $[a,b]$ with high probability (w.h.p.)
as $N \to \infty$, provided that the connectivity $\lm$ is also sufficiently large.
This is the content of Theorem \ref{Thm1}.

For technical reasons, we need a strengthened version of the overlap gap property, namely, the overlap gap
property for \emph{coupled} diluted $K$-spin models. Specifically, we consider two coupled copies 
$\G_{N,t}^j\sim \mathrm{ER}(K, \lm, N), j=1,2$, where $t\in [0,1]$ is a fixed parameter indicating some fraction of common edges in the two graphs. The two graphs are identical when $t=1$ and are independent when $t=0$. The details of the construction are given in  Section~\ref{sec:coupling}. We establish the overlap gap property for this coupled model with respect to the aforementioned interval $[a,b]$. Namely,
there exists an $\eta_2>0$ such that for all sufficiently large $\lm$  and for all $t\in [0,1]$ the following holds:
for every $\eta_2$-optimal $\sigma_1$ for $\G_{N,t}^1$ and every $\eta_2$-optimal $\sigma_2$ for $\G_{N,t}^2$, the absolute
value of the overlap between $\sigma_1$ and $\sigma_2$ does not take values in $[a,b]$ w.h.p. as $N\to\infty$.
This is the content of Theorem~\ref{prop4}.

The overlap gap property for the coupled model discussed above is derived as a corollary of a similar overlap gap property,
with the same interval $[a,b]$ as above, for the mean field $K$-spin model \eqref{eqn:HSK}.
The overlap gap property of the latter model is stated as Theorem~\ref{Thm2} and proved in Section \ref{sec:Thm2proof}
by using the Parisi formula at zero temperature for this model. The connection between the overlap gap property for the graph model of interest
and the overlap gap property for the mean field model above is established via an interpolation technique
that connects their respective nearly optimal spin configurations. The interpolation technique is also used
in~\cite{ChenP17, DMS, JKS, PKsat, Sen}. This is the content of Section \ref{sec:Thm1proof}.

Finally, the overlap gap property for the coupled model is used to establish limits on the largest cut achievable by local algorithms (defined
as factors of i.i.d). The proof is by contradiction, assuming that a local algorithm exists that provides a spin configuration achieving
some $\eta$-optimal cut value in expectation. Standard concentration results can be used to show that this construction
leads to an $\eta$-optimal cut w.h.p. as $N\to\infty$. Then we proceed by constructing two $\eta$-optimal and
coupled spin configurations $\sg^1$ and $\sg^2$, based on factors of i.i.d., on the two $t$-coupled graphs $\G_{N,t}^1$ and $\G_{N,t}^2$, respectively. Here we use $t$-coupled sources of i.i.d.~decorations on the nodes of the coupled graph.
In particular, when $t=0$, the decorations are independent, and, since the graphs are independent as well,
$\sg^1$ and $\sg^2$ are uncorrelated and have near zero asymptotical overlap. For this part, we also establish
that every nearly optimal spin configuration is nearly balanced w.h.p.; see Lemma \ref{prop3}.

Next we show that the overlap of $\sg^1$ and $\sg^2$ is a continuous function of $t$.
When $t=1$, not only are the graphs identical but the spin configurations are as well
since they are based on the same set of i.i.d.~decorations. Therefore their overlap is 1 when $t=1$.
The continuity in $t$ along with the endpoint values implies that for some $t$ the
overlap of $\sg^1$ and $\sg^2$ falls into the interval $[a,b]$. Thus, the cut values of $\sg^1$ and
$\sg^2$ will not be $\eta$-optimal with the $\eta$ referenced in Theorem~\ref{prop4}. This is the contradiction.

Let us explain why we needed to establish the overlap gap property for the coupled model.
If we proceeded with the same idea of $t$-coupled i.i.d.~decorations on the \emph{same} graph, we could not necessarily argue
that the resulting spin configurations $\sg^1$ and $\sg^2$ have a small overlap when $t=0$ since they would not be uncorrelated. In principle,
they could have a non-trivial overlap due to the inherent randomness of neighborhood structures of a typical vertex in
the $\mathrm{ER}(K, \lm, N)$ model.

To conclude, we remark that the properties of local algorithms that were utilized in the proof of
the main theorem are the following. First, concentration of the cut density and magnetization of
spin configurations generated by any local algorithm $f$. Second, concentration of the overlap,
$R_N(f,t)$, of the $t$-coupled spin configurations $\sg^1$ and $\sg^2$ generated from $f$ via our construction.
Third, the property that $t \to \e R_N(f,t)$ fills out a sufficiently dense subset of $[0,1]$ for large $N$.
Any class of algorithms that satisfy these properties will be sub-optimal in the sense of Theorem \ref{thm:laf}
(or the informal Theorem \ref{thm:lafinformal}).

\subsection{Overlap gap property} \label{sec:introgap}

The \emph{magnetization} of a spin configuration $\sg \in \Sigma^N$ is defined as
\begin{equation} \label{eqn:mag}
m(\sg) = \frac{1}{N}\sum_{i \in [N]} \sg_i\,.
\end{equation}
The \emph{overlap} of two spin configurations $\sg^1, \sg^2 \in \Sigma^N$ is defined as
\begin{equation} \label{eqn:overlap}
R_{1,2} = \frac{1}{N} \sum_{i=1}^N \sg^1_i \sg^2_i = \frac{\sg^1 \cdot \sg^2}{N}.
\end{equation}
As mentioned earlier, we need a statement of the aforementioned overlap gap property for coupled copies of the diluted $K$-spin model,
which is Theorem \ref{prop4} below. An already interesting illustration is the following special case.
Here, $\G_N \sim \mathrm{ER}(K, \lm, N)$ and $H_{\lambda, N}$ is its Hamiltonian as defined in (\ref{eqn:HgraphTwo}).
\begin{theorem} \label{Thm1}
	For any even $K\geq 4$, there exist $0 < a < b < 1$, $\eta_0 > 0$, $\lm_0>0$ and $N_0\geq 1$ such that, for $\lambda\geq \lambda_0$ and $N\geq N_0$, the following holds with probability at least $1-Le^{-N/L}$ for some $L=L(\eta_0)$: for any two spin configurations $\sigma^1,\sigma^2$ that satisfy 
	$$
	\frac{H_{\lambda, N}(\sg^{\ell})}{N} \geq \sqrt{\lambda}\bigl(P(K)-\eta_0\bigr), \; \ell=1,2,
	$$
	the absolute value of their overlap satisfies $|R_{1,2}| \notin [a,b]$.
\end{theorem}

The proof of Theorem \ref{Thm1} (or more general Theorem \ref{prop4} below) will based on approximating the diluted model in terms of the fully connected mean field model. The following theorem establishes the overlap gap property for the mean field $K$-spin Hamiltonian $H_N$ in (\ref{eqn:HSK}). Its proof is the most technical part of the paper.

\begin{theorem}\label{Thm2}
	For even $K\geq 4$, there exist $0<a<b<1$ and $\eta>0$ such that, for large enough $N$,
	\begin{equation}
	\label{Thm2:eq1}
	\frac{1}{N}\, \e\max_{\substack{\sg^1, \sg^2 \in \Sigma^N: \\ |R_{1,2}|\in [a,b]}} \, \bigl(H_N(\sg^1)+H_N(\sg^2)\bigr) 
	\leq \frac{2}{N}\, \e\max_{\sg \in \Sigma^N} H_N(\sg) - \eta.
	\end{equation}
\end{theorem}

The proof of Theorem \ref{Thm2} will utilize the zero temperature Parisi formula recently established by Auffinger and Chen in \cite{ChenAuf}, and the Guerra-Talagrand replica symmetry breaking bound at zero temperature, which is an extension of the corresponding bound at positive temperature used earlier by Guerra and Talagrand \cite{Guerra, TalUltra, TPF, SG2} to study the mean field $K$-spin model. Similar techniques have recently been used in \cite{BAJ} to study the questions of dynamics at positive temperature. For $K=2$, the case of ordinary graphs, the overlap gap property (\ref{Thm2:eq1}) is not expected to hold. See \cite{ACZ} for a discussion and recent rigorous results on the overlap structure of mean field spin models.

\paragraph{Outline of the paper}
Local algorithms are formally defined in Section \ref{sec:localg} and the main result is stated in Theorem \ref{thm:laf} following the necessary definitions. Section \ref{sec:hoa} contains the proof of Theorem \ref{thm:laf}. Section \ref{sec:Thm1proof} contains the proof of Theorem \ref{Thm1} and Theorem \ref{prop4}, assuming the validity of Theorem \ref{Thm2}. The proof of Theorem \ref{Thm2} is in Section \ref{sec:Thm2proof}.

\section{Definitions and main result} \label{sec:localg}

%We begin with some definitions. Formal statement of the main theorem, Theorem \ref{thm:laf}, is at the end of the section.

\paragraph{Hypergraphs and local algorithms:}
A $K$-uniform (directed) hypergraph is a pair $G = (V,E)$ where
$V$ is the set of vertices and $E \subseteq V^K$ are the edges,
each edge being a $K$-tuple of vertices. We write $v \in e$ ($v$ lies in $e$)
if the edge $e$ contains $v$ as a coordinate. Two vertices are said to be
connected if they lie in a common edge. The degree of a vertex is the
number of edges that contain it, counted with multiplicity as a vertex
may lie in an edge more than once. The distance between two vertices $v\not= w$ is
the minimal integer $d\geq 1$ such that there exists a sequence of vertices $v_1, \ldots, v_{d+1}$ with $v=v_1$ and $w=v_{d+1}$ and with any two consecutive $v_j, v_{j+1}$ for $1\leq j\leq d$ belonging to the same edge. Distance from $v$ to itself is zero. The $r$-neighbourhood
of $v$, denoted $B_r(G,v)$, is the subgraph of $G$ induced by all vertices within
distance $r$ of $v$. 

A rooted $K$-uniform hypergraph $(G,v)$ is a $K$-uniform hypergraph with a distinguised
vertex $v \in V(G)$ called the root. If $r$ is the smallest integer such that $B_r(G,v) = G$ then we
say $(G,v)$ \emph{has radius} $r$. Let 
\begin{align}
\begin{split}\label{eqn:Grr0}
\Gr_r &= \Bigl\{ (G,v) \,\mid\, (G,v) \mbox{ is a finite, connected,} \\
	&\qquad\mbox{$K$-uniform rooted hypergraphs with radius at most $r$}\Bigr\}.
\end{split}
\end{align}
Two rooted hypergraphs $(G,v)$ and $(G',v')$ are \emph{isomorphic}, denoted $(G,v)\cong (G',v')$, if there is a bijection $\phi\colon V(G) \to V(G')$ such that $\phi(v) = v'$ and $e \in E(G)$ if and only if $\phi(e) \in E(G')$, where $\phi$ acts on edges coordinate-wise. From now on, when we talk about hypergraphs, one can assume that we already picked a particular representative of an isomorphism class and think of isomorphisms as automorphisms of these representatives. In particular, we can redefine $\Gr_r$ as
\begin{equation}\label{eqn:Grr}
\Gr_r = \Bigl\{\mbox{representatives of isomorphism classes of (\ref{eqn:Grr0})}\Bigr\}.
\end{equation}
For a sequence of random rooted hypergraphs $(G_N, v_N)$, we define the sequence to \emph{convergence in distribution} in a natural way as convergence of the probabilities $\pr{B_r(G_N,v_N) \cong (G,v)}$ for every $r$ and every $(G,v)\in \Gr_r$. In this case, some elementary measure theoretic and
topological arguments imply that there is a random rooted graph $(G_{\infty}, v_{\infty})$ such that the aforementioned probabilities converge
to $\pr{B_r(G_{\infty},v_{\infty}) \cong (G,v)}$. This notion of convergence is often called local weak convergence of graphs or also
Benjamini-Schramm convergence; see \cite{HLS} and references therein.

\paragraph{Factors on rooted, labelled hypergraphs}
Let $\Lambda_r$ denote the collection of all \emph{$[0,1]$-labelled hypergraphs} of radius at most $r$, that is, the collection of all triples $(G,v,x)$, where
\begin{enumerate}
	\item $(G,v) \in \Gr_r$ defined in (\ref{eqn:Grr}).
	\item $x \in [0,1]^{V(G)}$ is a {labelling} of vertices of $G$.
\end{enumerate}
Two rooted $[0,1]$-labelled hypergraphs $(G,v,x)$ and $(G',v',x')$ are \emph{isomorphic} if there is a bijection $\phi\colon V(G) \to V(G')$, which is an isomorphism of $(G,v)$ and $(G',v')$, and such that
$$
x' = x\circ \phi^{-1}.
$$
A function $f \colon \Lambda_r \to \{-1,1\}$ is called a \emph{factor of radius $r$}, denoted $f\in\FF_r$, if it is measurable and 
$$
f(G,v,x) = f(G', v', x')
$$
for all isomorphic $(G,v,x)$ and $(G',v',x')$ in $\Lambda_r$. In other words, $f$ is constant on isomorphism classes of $\Lambda_r$. 

\paragraph{Local algorithms for $K$-uniform hypergraphs}
A local algorithm $A$ for $K$-uniform hypergraphs comes with an associated factor $f\colon \Lambda_r \to \{-1,1\}$. The input to $A$ is a finite $K$-uniform hypergraph $G$ and a \emph{random labeling} $X = (X(u))_{u \in V(G)}$ of $V(G)$ such that the $X(u)$'s are i.i.d. uniform random variables on $[0,1]$. The \emph{output} of the algorithm is a $\{-1,1\}$-valued stochastic process $\sg(f,X) =  (\sg(f,X)(u))_{u \in V(G)}$ defined by
\begin{equation} \label{eqn:AGv}
\sg(f,X)(u) = f \big (\, B_r(G,u), u,  \bigl(X(w)\bigr)_{w \in B_r(G,u)}\big).
\end{equation}
For simplicity of notation, when we write $w \in B_r(G,u)$ we really mean $w \in V(B_r(G,u)).$ Notice that the radius of $B_r(G,u)$ can be smaller than $r$, which is why a factor $f$ of radius $r$ was defined on hypergraphs of radius less than or equal to $r$.

One obvious property of factors that we will use implicitly throughout the paper is the following.
Suppose a factor $f \in \FF_r$ outputs the spin configuration $\sg(f,X)$ when run on a $K$-uniform hypergraph $G$.
Then for any subgraph $G' \subseteq G$, the distribution of the restriction of $\sg(X,f)$ onto $G'$ depends only
on the isomorphism class of the $r$-neighbourhood of $G'$ in $G$, i.e., on the subgraph of $G$ induced by $\cup_{v \in V(G')} B_r(G,v)$.
Below, we will use the same property for coupled hypergraph models without mentioning it.

\paragraph{Main result}
As above, $\G_N \Distr \mathrm{ER}(K,\lm,N)$ and $H_{\lambda, N}$ was defined in (\ref{eqn:HgraphTwo}).
Recall the definition of  the quantity $P(K)$ from \eqref{eqn:freeenergy}.

\begin{theorem} \label{thm:laf}
	Local algorithms defined as factors of i.i.d.~can not find spin configurations in $\mathrm{ER}(K,\lm,N)$ that are nearly optimal in the following sense.
	For even $K \geq 4$, there exists $\eta>0$ such that
	\begin{equation}\label{ThmMaineqn}
	\limsup_{\lm \to \infty}\, \sup_{r}\sup_{f\in\FF_r} \,
	\limsup_{N \to \infty}\; \e \frac{H_{\lambda, N}(\sg(f,X))}{\sqrt{\lm} N} \leq P(K)-\eta.
	\end{equation}\end{theorem}
The result establishes that, for large enough $\lm$, the performance of any local algorithm is bounded away from the optimal value $P(K)$ by some additive gap $\eta$ regardless of the choice of the factor $f$ and its depth $r$. By the concentration result in Lemma \ref{lemma:Factor-concentration} below, the same holds with high probability rather than on average.

\section{Proof of the main result} \label{sec:hoa}

\subsection{Concentration of cut density and magnetization} 
We begin by establishing a concentration of the cut density achievable by local algorithms.
\begin{lemma}\label{lemma:Factor-concentration}
	For every $\lm,r,f\in\FF_r$ and $\epsilon>0,$
	\begin{align*}
	\lim_{N\rightarrow\infty}
	\p\Bigl(\bigl|H_{\G_N}\bigl(\sg(f,X)\bigr)-\e H_{\G_N}\bigl(\sg(f,X)\bigr)\bigr|\ge \epsilon N\Bigr)=0.
	\end{align*}
\end{lemma}
\begin{proof}
	For each node $v\in [N]$, let $N_\ell(v)$ denote the set of edges $e$ containing $v$ in the $\ell$-th position of the 
	edge, for $\ell=1,2,\ldots,K$. 
	Let 
	\begin{align*}
	H_{\G_N,v}\bigl(\sg(f,X)\bigr)=
	\frac{1}{K} \, \sum_{1\le \ell\le K} \, \sum_{e\in N_\ell(v)} \, \theta\bigl(\sg(f,X),e\bigr). 
	\end{align*}
	Observe that $H_{\G_N}(\sg(f,X))=\sum_{v\in [N]}H_{\G_N,v}(\sg(f,X))$, since every edge $e$ appears in exactly $K$ terms. Let us expand the variance of $H_{\G_N}(\sg(f,X))$ in terms of $H_{\G_N,v}(\sg(f,X))$ for $v\in [N]$. We obtain,
	\begin{align*}
	&
	\e \Bigl(H_{\G_N}\bigl(\sg(f,X)\bigr)-\e H_{\G_N}\bigl(\sg(f,X)\bigr) \Bigr)^2 
	\\
	&= \sum_{v,u\in [N]}
	\e \Bigl(H_{\G_N,v}\bigl(\sg(f,X)\bigr)-\e H_{\G_N,v}\bigl(\sg(f,X)\bigr) \Bigr)\\
	&\qquad\qquad\quad
	\Bigl(H_{\G_N,u}\bigl(\sg(f,X)\bigr)-\e H_{\G_N,u}\bigl(\sg(f,X)\bigr)\Bigr).
	\end{align*}
	Let us fix any two distinct vertices $u$ and $v$ in $[N]$, for instance, $u=1$ and $v=2$.
	A classical fact from the theory of random graph states that for every fixed $r$ and $\lm$, as $N\to\infty$,
	$B_{r+1}(\G_N,u)\cap B_{r+1}(\G_N,v)=\emptyset$ with high probability and the pair $\left(B_{r+1}(\G_N,u),B_{r+1}(\G_N,v)\right)$
	converges in distribution to a pair of independent Galton-Watson processes both truncated at $r+1$ generations.
	In this limiting Galton-Watson process (hypertree), for each vertex, the number of offsprings (i.e.~hyperedges which contain this vertex
	and $K-1$ new distinct vertices that belong to the next generation) has $\pos(\lm K)$ distribution. Note that for every vertex $v$
	the value $H_{\G_N,v}(\sg(f,X))$ is determined by $B_{r+1}(\G_N,v)$ and the i.i.d.~labelling of the vertices in $B_{r+1}(\G_N,v)$.
	Therefore, we conclude that the random variables $H_{\G_N,v}(\sg(f,X))$ and $H_{\G_N,u}(\sg(f,X))$ are also asymptotically
	independent as $N\to\infty$. Both of these random variables are uniformly exponentially integrable in $N$, i.e., for every $t>0$,
	$\sup_N\E{\exp(tH_{\G_N,v}(\sg(f,X)))}<\infty$. Thus,
	\begin{align*}
	&\lim_{N\to\infty}
	\e \Bigl(H_{\G_N,v}\bigl(\sg(f,X)\bigr)-\e H_{\G_N,v}\bigl(\sg(f,X)\bigr) \Bigr)\\
	&\qquad\quad\Bigl(H_{\G_N,u}\bigl(\sg(f,X)\bigr)-\e H_{\G_N,u}\bigl(\sg(f,X)\bigr)\Bigr)
	=0.
	\end{align*}
	Notice that, by symmetry, this expectation is the same for all pairs of distinct vertices $u$ and $v$, and we conclude that the variance
	of $H_{\G_N}(\sg(f,X))$ satisfies
	\begin{equation*}
	\e \Bigl(H_{\G_N}\bigl(\sg(f,X)\bigr)-\e H_{\G_N}\bigl(\sg(f,X)\bigr) \Bigr)^2  = o(N^2)
	\end{equation*}
	as $N \to \infty$. Applying Chebyshev's inequality, we obtain the result.
\end{proof}

The identical proof also gives concentration of the magnetization.
\begin{lemma}\label{lemma:mag-concentration}
	For every $\lm,r,f\in\FF_r$ and $\epsilon>0,$
	\begin{align*}
	\lim_{N\rightarrow\infty}
	\p\Bigl(\bigl| m\bigl(\sg(f,X)\bigr)-\e m\bigl(\sg(f,X)\bigr)\bigr|\ge \epsilon N\Bigr)=0.
	\end{align*}
\end{lemma}
Notice also that, by symmetry, $\e m(\sg(f,X)) = \e \sg_1(f,X),$ where $\sigma_1$ is the first coordinate of $\sigma.$ In the proof of Theorem \ref{thm:laf}, we will need one more property of magnetization that we discuss next.

The proof of Theorem \ref{thm:laf} below will proceed by contradiction, assuming that, for any $\eta>0$, we can find arbitrarily large $\lambda$ for which there exists a factor $f\in\FF_r$ for some $r$ such that
\begin{equation}\label{contradiction}
\e \frac{H_{\G_N}(\sg(f,X))}{N} > \sqrt{\lambda}(P(K)-\eta)
\end{equation}
for some arbitrarily large $N$. The concentration inequality in Lemma \ref{lemma:Factor-concentration} then implies that, we can find arbitrarily large $N$ for which, with high probability,
\begin{equation}\label{contradiction2}
\frac{H_{\G_N}(\sg(f,X))}{N} > \sqrt{\lambda}(P(K)-\eta).
\end{equation}
It will be important to know that the magnetization $m(\sg(f,X))$ of the output $\sg(f,X)$ of the
algorithm is small for large $\lambda$, with high probability.
%Actually, it is straightforward to show, using the first moment method and Poisson concentration, that, with high probability, any configuration %with $H_{\G_N}(\sg)\geq 0$ must have absolute magnetization $|m(\sg)|$ of order $\lambda^{-1/(2K)},$ which would be enough for our %purposes. However,
For this, as a byproduct of the interpolation technique in Section \ref{sec:Thm1proof}, we will prove the following. 

\begin{lemma}
	\label{prop3} For any $\eta_0>0$, there exist $\lambda_0$ and $N_0$ such that for any $\lambda\geq \lambda_0$ and $N\geq N_0$, the following event holds with probability at least $1-Le^{-N/L}$ for some $L=L(\eta)$: Whenever $\sigma$ satisfies 
	\begin{align*}
	\frac{H_{\G_N}(\sigma)}{N}&> \sqrt{\lambda}\bigl(P(K)-\eta_0\bigr),
	\end{align*}
	it must also satisfy $|m(\sigma)|<\lambda^{-1/(2K)}(4\eta_0)^{1/K}.$
\end{lemma}
Take $\eta_0=1/8$ in the above.
In (\ref{contradiction}) and (\ref{contradiction2}), we can assume that $\eta\leq 1/8,$ $\lambda\geq \lambda_0$ and $N\geq N_0$ and, as a result, with high probability we get that the absolute magnetization $|m(\sg(f,X))|<\lambda^{-1/(2K)}2^{-1/K}$. Lemma \ref{lemma:mag-concentration} implies that $|\e m(\sg(f,X))|\leq \lambda^{-1/(2K)}$, although the condition of how large $N$ should be can now depend on $f$ and $\lambda.$ To summarize, we have the following.

\begin{lemma}\label{Lem-magnet}
	There exists $\lambda_0$ such that, for $\lambda\geq \lambda_0$, if (\ref{contradiction}) holds then $|\e m(\sg(f,X))|\leq \lambda^{-1/(2K)}$, for all large enough $N$.
\end{lemma}

\subsection{The coupled graph model and its properties} \label{sec:coupling}

The key idea in the proof of Theorem \ref{thm:laf} will be running the local algorithm on a pair of coupled hypergraphs
\begin{equation} \label{eqn:cplmodel}
\G_{N,t}^1 \,\,\mbox{ and }\,\, \G_{N,t}^2
\end{equation}
defined on the same set of $N$ vertices and indexed by the coupling parameter $t\in [0,1]$. The construction is as follows. Generate three independent random variables
$$|E_t| \Distr \pos(t \lm N) \;\;\text{and}\;\; |E_{t,1}| \;,\; |E_{t,2}| \Distr \pos((1-t)\lm N).$$
The somewhat unusual choice of notations $|E_{\cdot}|$ is motivated by having $E_{\cdot}$ denote
sets of edges in graphs constructed below, where $|E_{\cdot}|$ are the cardinalities of these sets. 
\begin{enumerate}
	\item For each $1 \leq i \leq |E_t|$, generate an edge $e_i \in [N]^K$ uniformly at random, independently for all $i$. The set of these edges is denoted by $E_t$. 
	
	\item For each $\ell = 1,2$ and each $1\le i\le |E_{t,\ell}|$ generate an edge $e_i \in [N]^K$ uniformly at random, independently for all $i,\ell$ and independently from the set $E_t$. For a given $\ell,$ the set of these edges is denoted $E_{t,\ell}.$
\end{enumerate}

The hypergraph $\G_{N,t}^\ell$ is defined by the edge set $E_t\cup E_{t,\ell}$ for $\ell=1,2$. Both $\G_{N,t}^1$ and $\G_{N,t}^2$ clearly have the same distribution $\mathrm{ER}(K,\lm,N)$. When $t=0$, the two hypergraphs are sampled independently of each other and, when $t=1$, the two are identical.

The following overlap gap property for the coupled random graph model will be crucial in the proof of Theorem \ref{thm:laf}. It will be proved in Section \ref{sec:Thm1proof}. The special case corresponding to $t=1$ was stated in Theorem \ref{Thm1} above.
\begin{theorem}\label{prop4} 
	For any even $K\geq 4$, there exist $0 < a < b < 1,$ $\eta_0 > 0$, $\lm_0>0$ and $N_0\geq 1$ such that, for any $t\in [0,1]$, $\lambda\geq \lambda_0$, $N\geq N_0$, the following holds with probability at least $1-Le^{-N/L}$ for some $L=L(\eta_0)$: whenever two spin configurations $\sigma^1,\sigma^2$ satisfy 
	$$
	\frac{H_{\G^{\ell}_N}(\sg^{\ell})}{N} \geq \sqrt{\lambda}\bigl(P(K)-\eta_0\bigr)
	$$
	then the absolute value of their overlap satisfies $|R_{1,2}| \notin [a,b]$.
\end{theorem}

To run local algorithms on the pair of hypergraphs $\G_{N,t}^1$ and $\G_{N,t}^2$ defined above, we need to define their labelings appropriately, as follows. Let us say that a vertex $v \in [N]$ is shared if it is incident to some shared edge in $E_t$, and non-shared otherwise. Consider a pair of independent random labellings
$X^1$ and $Y$ indexed by $[N]$. Define a new random labelling $X^2$ by
\begin{equation*}
X^2(v) = 
\begin{cases}
X^1(v), & \text{if}\; v \; \text{is shared} \\
Y(v), & \text{if}\; v \; \text{is non-shared.}
\end{cases}
\end{equation*}
Given a factor $f\in\FF_r$, let $\sg^\ell=\sg(f,X^\ell)$ be its output on $\G_{N,t}^\ell$ for $\ell=1,2$ and let
\begin{equation}
R_N(f,t) = \frac{1}{N}\sum_{i=1}^N \sg_i^1\sg_i^2
\end{equation}
be their overlap. 

In the proof of Lemma \ref{lemma:Factor-concentration}, we used a classical fact that a neighbourhood of a vertex in the Erd\H{o}s-R\'{e}nyi hypergraph converges to a Galton-Watson (hyper-)tree and, for any two distinct vertices, their neighbourhood configurations are asymptotically independent. Of course, the same result holds for coupled hypergraphs, where in the limit we get coupled Galton-Watson processes defined as follows. 

We start with a given vertex, call it a root, which is common to both trees. We generate Poisson$(t\lambda K)$ number of `shared' offsprings, where offspring of a vertex is a hyperedge containing this vertex and $K-1$ new vertices, and we also generate two independent Poisson$((1-t)\lambda K)$ sets of `non-shared' offsprings for each tree. The new vertices in `shared' edges are labelled `shared' and new vertices in `non-shared' edges are labelled `non-shared'. The root itself is labelled `non-shared' if 
Poisson$(t\lambda K) = 0$, otherwise, it is labelled `shared' (this will be important for the labelling process below). From the second generation on, any `non-shared' vertex produces Poisson$(\lambda K)$ number of only `non-shared' edges, and any `shared' vertex produces Poisson$(t\lambda K)$ number of `shared' edges and two independent Poisson$((1-t)\lambda K)$ sets of `non-shared' edges, one for each tree. 

Notice that the term `shared' does not correspond to the edges $E_t$ in the construction $\G_{N,t}^1$ and $\G_{N,t}^2$ above, but rather has a meaning of `disjoint' in a local sense. Namely, for a given vertex $v$, its neighbourhoods $B_r(\G_{N,t}^1,v)$ and $B_r(\G_{N,t}^2,v)$ can have edges belonging to $E_t$ for any $r$. However, locally, for a finite $r$, the vertices in these two neighbourhood are disjoint with high probability, unless they are all connected to $v$ through edges in $E_t$. In the sense of local neighbourhoods, in the limit `non-shared' corresponds to `disjoint' for finite $N$. In particular, since disjoint vertices are labelled by independent i.i.d.~labelings, the labelings of two Galton-Watson trees will be the same for `shared' vertices and independent for `non-shared' vertices.

First of all, since for any two distinct vertices $u,v\in [N],$ their neighbourhoods $(B_r(\G_{N,t}^1,v), B_r(\G_{N,t}^2,v))$ and $(B_r(\G_{N,t}^1,u), B_r(\G_{N,t}^2,u))$ are asymptotically independent labelled coupled Galton-Watson processes with $r$ generation described above, the proof identical to Lemma \ref{lemma:Factor-concentration} gives the following.
\begin{lemma}\label{lemma:overlap-concentration}
	For every $\lm,r,f\in\FF_r$, $t\in[0,1]$ and $\epsilon>0,$
	\begin{align*}
	\lim_{N\rightarrow\infty}
	\p\Bigl(\bigl|R_N(f,t) -\e R_N(f,t) \bigr)\bigr|\ge \epsilon \Bigr)=0.
	\end{align*}
\end{lemma}
Next, if we write, by symmetry, $\e  R_N(f,t) = \e \sigma_1^1 \sigma_1^2$ then the convergence of $(B_r(\G_{N,t}^1,1),B_r(\G_{N,t}^2,1))$ with its labelings to labelled coupled Galton-Watson process implies that 
\begin{equation}\label{GW-overlap}
\lim_{N\to\infty} \e  R_N(f,t) = \e \sigma^1 \sigma^2 \triangleq R(f,t),
\end{equation}
where $\sigma^1$ and $\sigma^2$ are outputs of the factor $f$ at the root of these two Galton-Watson trees. Since there are countably many configurations for this coupled tree process and, obviously, the probability of each configuration is continuous in $t$, we have the following.
\begin{lemma}\label{lemma:overlap-continuity}
	The function $R(f,t)$ is continuous in $t.$
\end{lemma}
Notice that $R(f,1)=1$, since for $t=1$ the two trees are completely `shared' with probability one, and $R(f,0)=(\e \sigma^1)^2$, since for $t=0$ the trees are completely independent, or `non-shared', with probability one. Moreover, $\e \sigma^1$ is the limit of magnetization,
\begin{equation}\label{mag-limit}
\lim_{N\to\infty} \e  m(\sigma(f,X)) = \lim_{N\to\infty}\e  \sigma_1(f,X) = \e \sigma^1.
\end{equation}
With this, we are ready to prove Theorem~\ref{thm:laf}.

\subsection{Completing the proof of Theorem~\ref{thm:laf}}

Consider $a,b,\eta_0$ defined in Theorem \ref{prop4} and let $\lambda_0$ be the larger of the two values $\lambda_0$ defined in Theorem \ref{prop4} and Lemma \ref{Lem-magnet}. 

The proof of Theorem \ref{thm:laf} below will proceed by contradiction. Assume that, for $\eta=\eta_0/2$, we can find 
$$
\lambda\geq \max(\lambda_0, a^{-K}),
$$
for which there exists a factor $f\in\FF_r$ for some $r$ such that
\begin{equation}\label{contradiction3}
\e \frac{H_{\G_N}(\sg(f,X))}{N} > \sqrt{\lambda}(P(K)-\eta)
\end{equation}
for arbitrarily large $N$. Lemma \ref{Lem-magnet} together with (\ref{mag-limit}) implies that $|\e \sigma^1|\leq \lambda^{-1/(2K)}$ and, therefore, 
$$
R(f,0)=(\e \sigma^1)^2\leq \lambda^{-1/K}<a.
$$ 
Since $R(f,t)$ in (\ref{GW-overlap}) is continuous, $R(f,0)<a$ and $R(f,1)=1$, there exists $t\in (0,1)$ such that $R(f,t)=(a+b)/2.$ For this $t$, (\ref{GW-overlap}) and Lemma \ref{lemma:overlap-concentration} imply that, with high probability, the overlap $R_N(f,t)$ will belong to the interval $(a,b)$ for large enough $N.$ On the other hand, (\ref{contradiction3}) and Lemma \ref{lemma:Factor-concentration} imply that 
$$
\frac{H_{\G^{\ell}_N}(\sg^{\ell})}{N} \geq \sqrt{\lambda}\bigl(P(K)-\eta_0\bigr),
$$
with probability going to one as $N\to\infty.$ By Theorem \ref{prop4}, with high probability we must have that $R_N(f,t)$ does not belong to the interval $[a,b]$, which is a contradiction. \qed

\section{Overlap gap in the diluted model} \label{sec:Thm1proof}

This section proves Theorem \ref{prop4}, as well as Theorem \ref{Thm1} as a simpler illustration. In order to prove these results, we will first use the Guerra-Toninelli interpolation from \cite{GT} to connect the diluted $K$-spin model to the mean field $K$-spin model. This connection has been used several times in the literature \cite{DMS, Sen, PKsat, JKS, ChenP17}. Theorem \ref{Thm1} will then be deduced from the conclusion of Theorem \ref{Thm2}.

Before starting with the proof we set some notation that will be used throughout the remaining Sections \ref{sec:Thm1proof} and \ref{sec:Thm2proof}. For $\lm > 0$ let $H_{\lm}$ denote the Hamiltonian of the diluted $K$-spin model with connectivity $\lm$ on $N$ vertices, that is, the $K$-spin Hamiltonian for the hypergraph $\mathrm{ER}(K,\lm,N)$. Recall the Hamiltonians $H_{\G_{N,t}^1}$ and $H_{\G_{N,t}^2}$ from \eqref{eqn:cplmodel}. Denote them respectively by $H_\lambda^1$ and $H_{\lambda}^2$ (we keep the dependence on the coupling parameter $t$ implicit). We also denote by $H$ the Hamiltonian of the mean field, or fully connected, $K$-spin model on $N$ vertices as defined in \eqref{eqn:HSK}.  Let $H'$ and $H''$ be independent copies of $H.$ For $t\in [0,1]$, set 
\begin{align}
\begin{split}
\label{label:comeback}
H^1&=\sqrt{t}H+\sqrt{1-t}H',
\\
H^2&=\sqrt{t}H+\sqrt{1-t}H''.
\end{split}
\end{align}
We will assume throughout that $N$ is even and suppress the dependence on $t$ and $N$ for clarity. 

First we establish a key interpolation lemma. Let $A\subseteq \Sigma^N,$ $S\subseteq [0,1]$, $t\in[0,1]$ and $\lambda>0$ be fixed.
For $s\in [0,1]$, consider the interpolating Hamiltonian
\begin{equation}
H(s,\sigma^1,\sigma^2) = \sum_{\ell=1}^2 \bigl(\delta H_{\lambda(1-s)}^\ell(\sigma^\ell) +\sqrt{s}\beta H^\ell(\sigma^\ell)\bigr),
\label{Hampert}
\end{equation}
where the Hamiltonians $H_{\lambda(1-s)}^\ell$ correspond to the coupled diluted $K$-spin model with the coupling parameter $t$ and connectivity parameter $\lambda(1-s)$. The parameters $\delta>0$ and $\beta>0$ will be chosen later. Let
\begin{equation}
\varphi(s) = \frac{1}{N}\, \e \log \, \sum_{\sg^1,\sg^2 \in A : |R_{1,2}| \in S} \exp H(s,\sigma^1,\sigma^2).
\label{varphis}
\end{equation}
Let us denote by $\la\,\cdot\,\ra_s$ the average with respect to the Gibbs measure
\begin{equation}
G_s(\sigma^1,\sigma^2) = \frac{\exp H(s,\sigma^1,\sigma^2)}{\sum_{{\sg^1,\sg^2 \in A : |R_{1,2}| \in S}} \exp H(s,\sigma^1,\sigma^2)}
\label{Gees}
\end{equation}
defined on $\{(\sigma^1,\sigma^2)\in A\times A \,:\, |R_{1,2}|\in S\}$. The following holds.
\begin{lemma} \label{lem3}
	For any $A\subseteq \Sigma^N,$ $S\subseteq [0,1]$, $t\in[0,1]$ and $\lambda>0,$ 
	\begin{align}
	\begin{split}
	\label{lem3:eq1}
	&\frac{1}{N}\e\max_{\sg^1,\sg^2 \in A : |R_{1,2}| \in S}\bigl(H_\lambda^1(\sigma^1)+H_\lambda^2(\sigma^2)\bigr)
	\\
	&=\,\, \frac{\sqrt{\lambda}}{N} \, \e\max_{\sg^1,\sg^2 \in A : |R_{1,2}| \in S}\bigl(H^1(\sigma^1)+H^2(\sigma^2)\bigr)\\
	&\qquad -\lambda\int_0^1\e \bigl\la m(\sigma^{1})^K+m(\sigma^{2})^K\bigr\ra_sds
	+O\bigl(\lambda^{1/3}\bigr),
	\end{split}
	\end{align}
	where $O\bigl(\lambda^{1/3}\bigr)\leq L \lambda^{1/3}$ for some $L>0$ independent of $N,A,S,t,\lambda.$
\end{lemma}

\begin{proof} The proof will make use of the interpolating Hamiltonian \eqref{Hampert} with appropriate choices for $\beta$ and $\delta$ in terms of $\lambda$.
	First, let us compute the derivative $\varphi'(s)=\mathrm{I}+\mathrm{II}$ as a sum of two terms coming from the Gaussian integration by parts and Poisson integration by parts. Let us denote the i.i.d.~samples (replicas) from $G_s$ by $(\sigma^{j,1},\sigma^{j,2})$ for $j \geq 1$, and denote the overlap
	$$
	R_{\ell,\ell'}^{j,j'} = \frac{1}{N}\sum_{i=1}^N \sigma_i^{j,\ell}\sigma_i^{j',\ell'}.
	$$
	Taking the derivative in $\sqrt{s}$ and using Gaussian integration by parts (see e.g.~\cite[Section 1.2]{SKmodel}) gives
	$$
	\mathrm{I} = \frac{\beta^2}{2}\Bigl(2+ 2t\e\bigl\la (R_{1,2}^{1,1})^K\bigr\ra_s
	- \e\bigl\la (R_{1,1}^{1,2})^K+(R_{2,2}^{1,2})^K+t(R_{1,2}^{1,2})^K+t(R_{2,1}^{1,2})^K\bigr\ra_s\Bigr).
	$$
	The rest of the calculation below is as in the proof of the Franz-Leone upper bound for the free energy in diluted models in \cite{FL, PT}. In order to differentiate $\varphi(s)$ with respect to the Poisson random variables involved, we use that 
	$$
	\frac{d}{d s}\e f\big(\pos(s)\big) = \e f\big (\pos(s)+1 \big) - \e f\big (\pos(s) \big)
	$$
	for a Poisson random variable $\pos(s)$ having mean $s$. If a function depends on several independent Poisson random variables, by the product rule, the derivative will be a sum of such differences, where $+1$ is added to one Poisson random variable at a time. Hence, with the notation 
	$\sum :=\sum_{\sg^1,\sg^2 \in A : |R_{1,2}| \in S},$ 
	\begin{align*}
	\mathrm{II} &=  -t\lambda\Bigl(\e\log \sum  \exp H^+(s,\sigma^1,\sigma^2)
	- \e\log \sum\exp H(s,\sigma^1,\sigma^2)\Bigr)\\
	&-(1-t)\lambda\Bigl(\e\log \sum  \exp H_1^+(s,\sigma^1,\sigma^2)
	- \e\log \sum\exp H(s,\sigma^1,\sigma^2)\Bigr)\\
	&-(1-t)\lambda\Bigl(\e\log \sum  \exp H_2^+(s,\sigma^1,\sigma^2)
	- \e\log \sum\exp H(s,\sigma^1,\sigma^2)\Bigr),
	\end{align*}
	where $H^+(s,\sigma^1,\sigma^2)$ includes one additional common edge $(i_1,\ldots,i_K)$ for both configurations $\sigma^1,\sigma^2$ (this new edge is chosen independently and uniformly at random), and $H_1^+(s,\sigma^1,\sigma^2)$ and $H_2^+(s,\sigma^1,\sigma^2)$ include an additional edge on either the first or second configurations respectively. That is,
	\begin{align*}
	H^+(s,\sigma^1,\sigma^2)&= H(s,\sigma^1,\sigma^2) -\delta \sigma_{i_{1}}^1\cdots\sigma^1_{i_{K}} -\delta \sigma^2_{i_{1}}\cdots \sigma^2_{i_{K}},\\
	H_1^+(s,\sigma^1,\sigma^2)&= H(s,\sigma^1,\sigma^2) -\delta \sigma_{i_{1}}^1\cdots\sigma^1_{i_{K}},\\
	H_2^+(s,\sigma^1,\sigma^2)&= H(s,\sigma^1,\sigma^2) -\delta \sigma_{i_{1}}^2\cdots\sigma^2_{i_{K}}.
	\end{align*}
	We can then rewrite the above as
	\begin{align*}
	\mathrm{II}& =  -t\lambda \e\log \bigl\la
	\exp\bigl(- \delta \sigma^1_{i_{1}}\cdots \sigma^1_{i_{K}} -\delta \sigma^2_{i_{1}}\cdots \sigma^2_{i_{K}}\bigr)\bigr\ra_s\\
	&-(1-t)\lambda \e\log\bigl\la \exp\bigl(-\delta \sigma_{i_1}^1\cdots\sigma_{i_K}^1\bigr)\bigr\ra_s-(1-t)\lambda \e\log\bigl\la \exp\bigl(-\delta \sigma_{i_1}^2\cdots\sigma_{i_K}^2\bigr)\bigr\ra_s.
	\end{align*}
	Since the product of spins takes values $\pm 1$, we can represent
	$$
	\exp\bigl( - \delta \sigma^\ell_{i_{1}}\cdots \sigma^\ell_{i_{K}} \bigr)
	=
	\ch(\delta)\bigl(1-\myth(\delta)\sigma^\ell_{i_{1}}\cdots \sigma^\ell_{i_{K}}\bigr),
	$$
	and
	$$
	\exp\bigl(- \delta \sigma^1_{i_{1}}\cdots\sigma^1_{i_{K}} - \delta \sigma^2_{i_{1}}\cdots \sigma^2_{i_{K}} \bigr)
	= 
	\ch(\delta)^2\bigl(1-\myth(\delta)\Delta(\sigma^1,\sigma^2)\bigr)
	$$
	with the notation
	$$
	\Delta(\sigma^1,\sigma^2) = \sigma^1_{i_{1}}\cdots\sigma^1_{i_{K}} + \sigma^2_{i_{1}}\cdots \sigma^2_{i_{K}}
	-\myth(\delta)\sigma^1_{i_{1}}\cdots\sigma^1_{i_{K}}\sigma^2_{i_{1}}\cdots \sigma^2_{i_{K}},
	$$
	where $\ch()$ and $\myth()$ are hyperbolic cosine and tangent, respectively. 
	Expressing the logarithm by its Taylor series, we can rewrite
	$$
	\mathrm{II}=
	-2\lambda\log\ch\delta
	+\lambda \sum_{n\geq 1}\frac{\myth(\delta)^n}{n}\Bigl(t\e \bigl\la \Delta(\sigma^1,\sigma^2) \bigr\ra_s^n+(1-t)\sum_{\ell=1}^2\e\bigl\la \sigma_{i_1}^\ell\cdots\sigma_{i_K}^\ell\bigr\ra_s^n\Bigr).
	$$
	Using replicas, we can represent 
	$$
	\e \bigl\la \Delta(\sigma^1,\sigma^2) \bigr\ra_s^n
	=
	\e \bigl\la \prod_{j\leq n}\Delta(\sigma^{j,1},\sigma^{j,2}) \bigr\ra_s
	=
	\e \bigl\la \e' \prod_{j\leq n}\Delta(\sigma^{j,1},\sigma^{j,2}) \bigr\ra_s,
	$$
	where $\e'$ is the expectation with respect to the random indices $i_1, \ldots, i_K$, which are independent of the random
	variables in $\la\,\cdot\,\ra_s$. For $n=1$,
	$$
	\e'\sigma_{i_1}^\ell\cdots\sigma_{i_K}^\ell=m(\sigma^\ell)^K
	$$
	and
	$$
	\e' \Delta(\sigma^{1,1},\sigma^{1,2})
	= m(\sigma^{1,1})^K+m(\sigma^{1,2})^K-\myth(\delta) (R_{1,2}^{1,1})^K.
	$$
	Therefore, 
	\begin{align*}
	&\,\,\,
	t\e \bigl\la \Delta(\sigma^1,\sigma^2) \bigr\ra_s
	+(1-t)\e\bigl\la \sigma_{i_1}^1\cdots\sigma_{i_K}^1\bigr\ra_s
	+(1-t)\e\bigl\la \sigma_{i_1}^2\cdots\sigma_{i_K}^2\bigr\ra_s
	\\
	=&\,\,\,
	t\Bigl(\e \bigl\la m(\sigma^{1,1})^K+m(\sigma^{1,2})^K\bigr\ra_s-\myth(\delta) \e\bigl\la(R_{1,2}^{1,1})^K\bigr\ra_s\Bigr)
	\\
	&\,\,\,
	+(1-t)\e \bigl\la m(\sigma^{1,1})^K+m(\sigma^{1,2})^K\bigr\ra_s
	\\
	=&\,\,\,
	\e \bigl\la m(\sigma^{1,1})^K+m(\sigma^{1,2})^K\bigr\ra_s-t\myth(\delta) \e \bigl\la (R_{1,2}^{1,1})^K \bigr\ra_s
	\end{align*}
	For $n=2$, we compute similarly,
	$$
	\e \bigl\la \Delta(\sigma^1,\sigma^2) \bigr\ra_s^2
	=
	\e\bigl\la
	(R_{1,1}^{1,2})^K+ (R_{2,2}^{1,2})^K+ (R_{1,2}^{1,2})^K + (R_{2,1}^{1,2})^K
	\bigr\ra_s + \mathrm{III_1}
	$$
	and
	$$
	\sum_{\ell=1}^2\e\bigl\la \sigma_{i_1}^\ell\cdots\sigma_{i_K}^\ell\bigr\ra_s^2=\e\bigl\la \bigl(R_{1,1}^{1,2}\bigr)^K+\bigl(R_{2,2}^{1,2}\bigr)^K\bigr\ra_s,
	$$
	where $|\mathrm{III_1}|\leq L\myth(\delta).$ Finally, since $|\Delta(\sigma^1,\sigma^2)|\leq 2$ and $|\sigma_{i_1}^\ell\cdots\sigma_{i_K}^\ell|=1$, we can bound
	$$
	\Bigl|t\lambda \sum_{n\geq 3}\frac{\myth(\delta)^n}{n}
	\e \bigl\la \Delta(\sigma^1,\sigma^2) \bigr\ra_s^n\Bigr|
	\leq
	\lambda \sum_{n\geq 3}\frac{(2\myth(\delta))^n}{n}
	\leq L \lambda \delta^3
	$$
	and
	$$
	\Bigl|(1-t)\lambda \sum_{n\geq 3}\frac{\myth(\delta)^n}{n}\sum_{\ell=1}^{2}\e\bigl\la \sigma_{i_1}^\ell\cdots\sigma_{i_K}^\ell\bigr\ra_s^n\Bigr|\leq L\lambda \delta^3
	$$
	for small enough $\delta>0,$ where $L$ is a universal constant independent of everything. Collecting all the terms,
	\begin{align*}
	\mathrm{II}=&\,\,
	-2\lambda\log\ch\delta
	+\lambda \myth(\delta)\e \bigl\la m(\sigma^{1,1})^K+m(\sigma^{1,2})^K\bigr\ra_s-t\lambda \myth(\delta)^2 \e \bigl\la (R_{1,2}^{1,1})^K \bigr\ra_s
	\\ &
	+\frac{\lambda \myth(\delta)^2}{2}\e\bigl\la(R_{1,1}^{1,2})^K+(R_{2,2}^{1,2})^K+t(R_{1,2}^{1,2})^K+t(R_{2,1}^{1,2})^K\bigr\ra_s + O(\lambda\delta^3).
	\end{align*}
	Next, for a given $\lambda$ and $\delta$, we will define $\beta$ by
	\begin{equation}
	\beta = \sqrt{\lambda}\myth(\delta).
	\end{equation}
	With this choice, all the terms in $\varphi'(s) = \mathrm{I} + \mathrm{II}$ containing the overlaps $R_{\ell,\ell'}^{j,j'}$ cancel out and we get
	$$
	\varphi'(s)= \mathrm{I}+\mathrm{II}= -2\lambda\log\ch\delta + \lambda \myth(\delta)\e \bigl\la m(\sigma^{1,1})^K+m(\sigma^{1,2})^K\bigr\ra_s+\lambda\myth(\delta)^2+O(\lambda\delta^3).
	$$
	One can check that $-2\log\ch\delta + \myth(\delta)^2=O(\delta^4)$ as $\delta\to 0$ and, therefore, $$
	\varphi'(s)=\lambda \delta \e \bigl\la m(\sigma^{1,1})^K+m(\sigma^{1,2})^K\bigr\ra_s+O(\lambda\delta^3).$$ Now, integrating between $0$ and $1$ and dividing both sides by $\delta$,
	$$
	\frac{1}{\delta}\varphi(1)-\frac{1}{\delta} \varphi(0) =\frac{1}{\delta}\int_0^1\phi'(s)ds=\lambda \int_0^1\e \bigl\la m(\sigma^{1,1})^K+m(\sigma^{1,2})^K\bigr\ra_sds+O(\lambda \delta^2).
	$$
	Using the elementary estimates
	$$
	\frac{1}{N}\e\max_{{\sg^1,\sg^2 \in A : |R_{1,2}| \in S}} H(s,\sigma^1,\sigma^2) \leq \varphi(s) \leq 2\log 2 + \frac{1}{N}\e\max_{{\sg^1,\sg^2 \in A : |R_{1,2}| \in S}} H(s,\sigma^1,\sigma^2)
	$$
	at $s=0$ and $s=1$, we infer from the above that
	\begin{align*}
	\Bigl|\frac{1}{N}\e\max_{{\sg^1,\sg^2 \in A : |R_{1,2}| \in S}}\bigl(H_\lambda^1(\sigma^1)+H_\lambda^2(\sigma^2)\bigr) -  \frac{1}{\delta} \varphi(0)\Bigr|& \leq \frac{2\log 2}{\delta},\\
	\Bigl| \frac{\beta}{\delta} \frac{1}{N}\e\max_{{\sg^1,\sg^2 \in A : |R_{1,2}| \in S}}\bigl(H^1(\sigma^1)+H^2(\sigma^2)\bigr) 
	- \frac{1}{\delta} \varphi(1) \Bigr|& \leq \frac{2\log 2}{\delta}.
	\end{align*}
	By Taylor's expansion and our choice of $\beta$,
	$$
	\frac{\beta}{\delta}= \frac{\sqrt{\lambda}\myth(\delta)}{\delta} = \sqrt{\lambda} +O(\sqrt{\lambda}\delta^2)
	$$
	and, therefore,
	\begin{align*}
	&\frac{1}{N}\e\max_{{\sg^1,\sg^2 \in A : |R_{1,2}| \in S}}\bigl(H_\lambda^1(\sigma^1)+H_\lambda^2(\sigma^2)\bigr)\\
	&= 
	\frac{\sqrt{\lambda}}{N}\e\max_{{\sg^1,\sg^2 \in A : |R_{1,2}| \in S}}\bigl(H^1(\sigma^1)+H^2(\sigma^2)\bigr) \\
	&\qquad -\lambda  \int_0^1\e \bigl\la m(\sigma^{1,1})^K+m(\sigma^{1,2})^K\bigr\ra_sds +O\Bigl(\frac{1}{\delta}+ \lambda \delta^2+\sqrt{\lambda}\delta^2\Bigr).
	\end{align*}
	Letting $\delta=\lambda^{-1/3}$ in the last equation completes our proof.
\end{proof}

Using the above interpolation and Theorem \ref{Thm2} (which will be proved in the following section), we are now ready to prove Theorem \ref{prop4}.

\begin{proof}[Proof of Theorem \ref{prop4}]
	Set $S=[a,b]$ as in Theorem \ref{Thm2} and define
	\begin{align*}
	f(t):=\limsup_{N\rightarrow\infty}\frac{1}{N} \, \e\max_{|R_{1,2}| \in S}\bigl(H^1(\sigma^1)+H^2(\sigma^2)\bigr),
	\end{align*}
	where $H^1, H^2$ were defined in (\ref{label:comeback}).
	It is easy to see that $f$ is a continuous function of $t.$ By Theorem \ref{Thm2} above and \cite[Theorem 2]{ChenP17} (used with $\eps=a/2$ there), we have that 
	$$
	f(t)< 2P(K) \,\,\mbox{ for all } t\in [0,1].
	$$ 
	
	[\emph{Remark:} In fact, the case of $K$-spin model was established in Theorem 2 in \cite{ChenHL} and stated for $t\in(0,1)$ in both \cite{ChenHL, ChenP17}. However, the case of $t=0$ is implicitly included in the proof of Theorem 2 in \cite{ChenHL} and is in fact the easiest case. One should simply note that, in the notation of that paper, the derivative of the Guerra-Talagrand upper bound for $t=0$ with respect to the Lagrange multiplier $\lambda$ equals minus the overlap constraint $-q$, which results in energy penalty for non-zero overlap constraints.]
	
	The continuity of $f$ ensures that there exists some $\eta_0>0$ such that $f(t)<2P(K)-6\eta_0$ for all $t\in [0,1].$ Using Lemma \ref{lem3} with $A=\Sigma^N$ and noting that the terms involving magnetization are negative, we see that, for any $t\in [0,1]$ and large enough $N\geq N_0$,
	\begin{align*}
	\frac{1}{N}\e\max_{|R_{1,2}| \in S}\bigl(H_\lambda^1(\sigma^1)+H_\lambda^2(\sigma^2)\bigr)
	&\leq \frac{\sqrt{\lambda}}{N} \, \e\max_{ |R_{1,2}| \in S}\bigl(H^1(\sigma^1)+H^2(\sigma^2)\bigr)+L \lambda^{1/3}\\
	&\leq \sqrt{\lambda}\bigl(2P(K)-6\eta_0\bigr)+L \lambda^{1/3}
	\end{align*}
	and, thus, if $\lambda\geq \lambda_0=(L\eta_0)^{-6}$ for some large enough $L$ then  
	\begin{align*}
	\frac{1}{N}\e\max_{\substack{\sg^1,\sg^2 \in A \,: \\ |R_{1,2}| \in S}}\bigl(H_\lambda^1(\sigma^1)+H_\lambda^2(\sigma^2)\bigr)\leq 2\sqrt{\lambda}\bigl(P(K)-2\eta_0\bigr).
	\end{align*}
	By a standard concentration inequality for Poisson$(\lambda N)$ and Azuma's concentration inequality, one can check that there exists $L=L(\eta_0)$ such that, with probability at least $1-Le^{-N/L}$,
	\begin{align*}
	&\frac{1}{N}\max_{|R_{1,2}| \in S}\bigl(H_\lambda^1(\sigma^1)+H_\lambda^2(\sigma^2)\bigr)\\
	&\leq
	\frac{1}{N}\e\max_{ |R_{1,2}| \in S}\bigl(H_\lambda^1(\sigma^1)+H_\lambda^2(\sigma^2)\bigr) +\eta_0\sqrt{\lambda}\\
	&\leq 2\sqrt{\lambda}\bigl(P(K)-\eta_0\bigr).
	\end{align*}
	On this event, if $\sigma^1,\sigma^2$ satisfy
	\begin{align*}
	\frac{H_\lambda^\ell(\sigma^\ell)}{N}> \sqrt{\lambda}(P(K)-\eta_0), \ell=1,2,
	\end{align*}
	then $|R(\sigma^1,\sigma^2)|\notin S=[a,b],$ which finishes the proof.
\end{proof}

Finally, we use the interpolation to prove Lemma \ref{prop3}, which shows that near maximizers automatically have small magnetization.

\begin{proof}[Proof of Lemma \ref{prop3}]
	Let $c=(4\eta_0)^{1/K}\lambda^{-1/(2K)}$. In the setting of Lemma \ref{lem3}, let $A=\bigl\{\sigma\in\Sigma^N:|m(\sigma)|\geq c\bigr\}$, $S=[0,1]$ and $t=1$. In this case, $H_\lambda^1=H_\lambda^2=H_\lambda$ and $H^1=H^2=H.$ Since
	\begin{align*}
	\lambda \int_0^1\e\bigl\la m(\sigma^1)^K+m(\sigma^2)^K\bigr\ra_sds\geq 2\lambda c^K=8\sqrt{\lambda}\eta_0,
	\end{align*} 
	we have
	\begin{align*}
	\frac{1}{N}\e\max_{\sigma\in A}H_\lambda(\sigma)
	&\leq 
	\frac{\sqrt{\lambda}}{N} \, \e\max_{\sigma\in A} H(\sigma)
	-4\sqrt{\lambda}\eta_0+L\lambda^{1/3}
	\end{align*}
	and, thus, there exists $N_0$ such that, for all $N\geq N_0$,
	\begin{align*}
	\frac{1}{N}\e\max_{\sigma\in A} H_\lambda(\sigma)
	&\leq \frac{\sqrt{\lambda}}{N} \, \e\max_{\sigma\in\Sigma^N} H(\sigma) -4\sqrt{\lambda}\eta_0+O\bigl(\lambda^{1/3}\bigr)\\
	&\leq \sqrt{\lambda}\bigl(P(K)-3\eta_0\bigr)+L \lambda^{1/3}.
	\end{align*}
	Consequently, for $N\geq N_0$ and $\lambda\geq L^6\eta_0^{-6}$, 
	\begin{align*}
	\frac{1}{N}\e\max_{\sigma\in A} H_\lambda(\sigma)
	&\leq \sqrt{\lambda}\bigl(P(K)-2\eta_0\bigr).
	\end{align*}
	From standard concentration inequality for Poisson$(\lambda N)$ and Azuma's concentration inequality, one can check that there exists $L=L(\eta_0)$ such that, with probability at least $1-Le^{-N/L}$,
	\begin{align*}
	\frac{1}{N}\max_{\sigma\in A} H_\lambda(\sigma)
	&\leq 
	\frac{1}{N}\e\max_{\sigma\in A} H_\lambda(\sigma) + \sqrt{\lambda}\eta_0
	\leq
	\sqrt{\lambda}\bigl(P(K)-\eta_0\bigr).
	\end{align*}
	This means that, on this event, any $\sigma$ satisfying
	\begin{align*}
	\frac{H_\lambda(\sigma)}{N}&> \sqrt{\lambda}\bigl(P(K)-\eta_0\bigr)
	\end{align*}
	must also satisfy $|m(\sigma)|<c=\lambda^{-1/(2K)}(4\eta_0)^{1/K}.$
\end{proof}

\section{Overlap gap in the mean field model} \label{sec:Thm2proof}

In this section we prove Theorem \ref{Thm2}. As mentioned in the Introduction, it will be proved by means
of the Guerra-Talagrand (GT) replica symmetry breaking bound at zero temperature, which is obtained
from the positive temperature bound \cite{Guerra, TPF, TalUltra, SG2} by an appropriate rescaling procedure.
We begin by recalling the analogue of the Parisi formula for the limit of the normalized maximum of the
Hamiltonian \eqref{eqn:HSK},
\begin{equation}
\ME_N:=\frac{1}{N}\max_{\sigma \in \{-1,1\}^N}\, H(\sigma),
\end{equation}
which was proved in \cite{ChenAuf}, as well as the zero-temperature Guerra-Talagrand bound for coupled copies.

\subsection{The Parisi formula and Guerra-Talagrand bound}

Let $\mathcal{U}$ be the set of all nonnegative nondecreasing right-continuous functions $\gamma$ on the interval $[0,1)$ such that $\int_0^1\gamma(s)ds<\infty.$ We equip the space $\mathcal{U}$ with the $L^1(dx)$ norm. Let 
\begin{equation}
\xi(s)=s^K
\end{equation}
and define the zero temperature analogue of the Parisi functional \cite{Parisi79, Parisi} on $\mathcal{U}$ by
\begin{align}\label{pf}
\mathcal{P}(\gamma)=\Phi_\gamma(0,0)-\frac{1}{2}\int_0^1s\xi''(s)\gamma(s)ds,
\end{align}
where $\Phi_\gamma(s,x)$ for $(s,x)\in[0,1)\times\mathbb{R}$ is the weak \footnote{Recall from \cite{ChenHL,JagTob} that $u$ is a weak solution to $\eqref{pde}$ if it has an essentially bounded weak derivative $\partial_xu$ and for every $\phi\in C_c^\infty((0,1]\times \mathbb{R})$, $\int_0^1\int_\mathbb{R}\Bigl(-u \partial_t\phi+\frac{\xi''(t)}{2}\bigl(u\partial_{xx}\phi+\gamma(s)(\partial_xu)^2\bigr)\Bigr)dxds+\int_\mathbb{R}\phi(1,x)|x|dx=0.$}
solution of the Parisi PDE
\begin{align}
\label{pde}
\partial_s\Phi_\gamma(s,x)&=-\frac{\xi''(s)}{2}\bigl(\partial_{xx}\Phi_{\gamma}(s,x)+\gamma(s)\bigl(\partial_x\Phi_\gamma(s,x)\bigr)^2\bigr)
\end{align}
with the boundary condition $\Phi_\gamma(1,x)=|x|.$ The existence, uniqueness, and regularity properties of the solution $\Phi_\gamma(s,x)$ were studied in \cite[Appendix]{ChenHL}. The Parisi formula for the maximum energy in \cite{ChenAuf} states
\begin{align}
\label{Parisi}
\ME&:=\lim_{N\rightarrow\infty}\ME_N=\inf_{\gamma\in \mathcal{U}}\mathcal{P}(\gamma),
\end{align}
where the limit of $\ME_N$ exists almost surely. Indeed, \eqref{Parisi} was established in \cite{ChenAuf} for general mixed $p$-spin models by showing how the corresponding formulas at positive temperature, first proved for mixed even $p$-spin models in \cite{TPF} and for general mixed $p$-spin models in \cite{PUltra, PPF}, are transformed in the zero-temperature limit. The minimizer on the right-hand side exists and is unique, as was shown in \cite{ChenHL} building upon the ideas from \cite{ChenAufP}. We will denote this minimizer by $\gamma_P$ and call $\gamma_P(ds)$ the Parisi measure. 

We will now describe the Guerra-Talagrand bound for two coupled systems, which is discussed in Section 3.1 of \cite{ChenHL}. For any $(\lambda,\bx)\in \mathbb{R}\times\mathbb{R}^2,$ set
\begin{align}
\begin{split}\label{gf}
g(\lambda,\bx)&=\max\bigl(x_1+x_2+\lambda,-x_1-x_2+\lambda,x_1-x_2-\lambda,-x_1+x_2-\lambda\bigr).
\end{split}
\end{align}
Fix $\gamma \in \mathcal{U}$, $q\in[0,1)$ and $\lambda\in\Reals$. First, we define $\Gamma_\gamma(\lambda,s,\bx)$ for $(s,\bx)\in [q,1]\times\mathbb{R}^2$ to be the weak solution of the following equation,
\begin{align}
\label{pde2}
\partial_s\Gamma_{\gamma}&=-\frac{\xi''(s)}{2}\bigl(\Delta \Gamma_\gamma+\gamma(s)\bigl\la\nabla \Gamma_{\gamma},\nabla \Gamma_{\gamma}\bigr\ra\bigr),
\end{align}
with the boundary condition at $s=1$ given by
\begin{align}\label{eq1}
\begin{split}
\Gamma_{\gamma}(\lambda,1,\bx)&=g(\lambda,\bx).
\end{split}
\end{align}
Here, $\Delta\Gamma_\gamma$ is the Laplacian and $\nabla \Gamma_\gamma$ is the gradient of $\Gamma_\gamma$ in $\bx$. The existence and uniqueness of the weak solution to this PDE can be established in a similar way as in the one-dimensional case \eqref{pde} without essential changes, see \cite{ChenHL} and also \cite[Appendix 1]{BAJ}. Next, using $\Gamma_\gamma(q,x,x)$ as a boundary condition at $s=q$, we define $\Psi_\gamma(\lambda,s,x)$ for $(s,x)\in [0,q]\times\mathbb{R}$ to be the solution of
\begin{align}\label{pde3}
\partial_s\Psi_\gamma&=-\frac{\xi''(s)}{2}\bigl(\partial_{xx}\Psi_\gamma+\gamma(s)\bigl(\partial_x\Psi_\gamma\bigr)^2\bigr)
\end{align}
with the boundary condition $\Psi_\gamma(\lambda, q,x)=\Gamma_\gamma(\lambda,q,x,x).$ The existence and regularity properties of $\Gamma_\gamma$ and $\Psi_\gamma$ were studied in \cite{ChenHL} and will also be discussed in the proof of Proposition \ref{lem2} below. Finally, for all $(\lambda,\gamma)\in\mathbb{R}\times\mathcal{U}$, we define
\begin{align}\label{GTfunc}
\mathcal{T}_q(\lambda,\gamma)=\Psi_\gamma(\lambda,0,0)-\lambda q-\Bigl(\int_0^1s\xi''(s)\gamma(s)ds+\int_0^{q}s\xi''(s)\gamma(s)ds\Bigr).
\end{align}
If, for a subset $A\subset[-1,1]$ we define the maximum coupled energy as
\begin{align}
\MCE_N(A)&=\frac{1}{N}\max_{R_{1,2}\in A}\bigl(H(\sigma^1)+H(\sigma^2)\bigr)
\end{align}
and let $S_N=\{R_{1,2}\,:\,\sigma^1,\sigma^2\in \{-1,+1\}^N\}$ be the set of feasible overlap values for a given $N$ then, for any $q\in S_N\cap[0,1],$
\begin{align}\label{RSB}
\e\MCE_N(\{q\})&\leq \mathcal{T}_q(\lambda,\gamma).
\end{align}
This zero temperature Guerra-Talagrand bound follows from the classical positive temperature bound (see \cite{TPF} or \cite[Section 14.6]{SG2}) by the same rescaling argument as in the setting of one system explained in Lemma 2 in \cite{ChenAuf}. By symmetry of the model for even $K$, the same bound holds for $\e\MCE_N(\{-q\})$ as well. Our goal now will be to show that the right hand side can be made strictly less than twice $\mathcal{P}(\gamma_P)$ defined in (\ref{Parisi}) for all $q$ in some non-trivial interval $(a,b)$, and this will be done by an appropriate choice of $\gamma$ in (\ref{RSB}) and slightly perturbing $\lambda$ around zero. The following subsection prepares needed auxiliary results.

\subsection{Variational representations for $\Phi_\gamma$ and $\Psi_\gamma$}

Observe that both $\Phi_\gamma$ and $\Psi_\gamma$ are special cases of the Hamilton-Jacobi-Bellman equation, induced by a linear problem of diffusion control. In these cases, it is well-known that they can be expressed as optimal stochastic control problems by means of dynamic programming \cite{FS}. More precisely, denote by $W=\{W(w),\mathcal{G}_w,0\leq w\leq 1\}$ a standard Brownian motion, where the filtration $(\mathcal{G}_w)_{0\leq w\leq 1}$ satisfies the usual conditions (see Definition~2.25 in Chapter 1 of \cite{KS}). For $0\leq q\leq 1,$ denote by $D[q]$ the space of all progressively measurable processes $u$ with respect to $(\mathcal{G}_w)_{0\leq w\leq q}$ satisfying $\sup_{0\leq w\leq q}|u(w)|\leq 2$. Endow the space $D[q]$ with the norm
\begin{align}\label{norm}
\|u\|_{q}&=\Bigl(\e\int_0^q\! u(w)^2\, dw\Bigr)^{1/2}.
\end{align}
The following variational representation holds for $\Phi_\gamma$. 

\begin{proposition}\label{cor1} Let $\gamma\in \mathcal{U}$. For any $0\leq s\leq q\leq 1$ and $x\in\Reals,$
	\begin{align}
	\begin{split}\label{cor1:eq1}
	\Phi_\gamma(s,x)&=\max_{u\in D[q]} \e\Bigl[\Phi_\gamma\Bigl(q,x+\int_s^q\xi''\gamma udw+\int_s^q\sqrt{\xi''}dW\Bigr)-\frac{1}{2}\int_s^q\xi''\gamma u^2dw\Bigr].
	\end{split}
	\end{align}
	The maximum is attained on $u_\gamma(w)=\partial_x\Phi_\gamma(w,X_\gamma(w)),$ where $X_\gamma=(X_\gamma(w))_{s\leq w\leq q}$ is the solution to the following SDE,
	\begin{align}
	\label{cor1:eq2}
	dX_\gamma&=\xi''(w)\gamma(w) \partial_x\Phi_\gamma(w,X_\gamma)dw+\sqrt{\xi''(w)}dW,
	\end{align} 
	with the initial condition $X_\gamma(s)=x$.
\end{proposition}
Similarly, the functional $\Psi_\gamma(\lambda,0,0)$ in \eqref{GTfunc} can be written via the following optimal stochastic control problem. Recall that $\Psi_\gamma(\lambda, q,x)=\Gamma_\gamma(\lambda,q,x,x).$

\begin{proposition}\label{prop2} Let $(\lambda,\gamma)\in\mathbb{R}\times\mathcal{U}.$ For any $0\leq s\leq q\leq 1$ and $x\in\Reals,$
	\begin{align}
\begin{split}	\label{prop2:eq2}
	\Psi_\gamma(\lambda,s,x)&=\max_{u\in D[q]}\e\Bigl[\Psi_\gamma\Bigl(\lambda,q,x+\int_s^q\xi''\gamma udw+\int_s^q\sqrt{\xi''}d{W}\Bigr)\\
	&\qquad\qquad\qquad\qquad-\frac{1}{2}\int_s^q\xi''\gamma u^2dw\Bigr].
	\end{split}
	\end{align}
	The maximum is attained on $u_{\gamma}^\lambda(w)=\partial_x \Psi_\gamma(\lambda, w,X_{\gamma}^\lambda(w))$, where $(X_{\gamma}^\lambda(w))_{s\leq w\leq q}$ is the solution of the following SDE,
	\begin{align}
	%\begin{split}
	\label{cor1:eq22}
	dX_{\gamma}^\lambda&=\xi''(w)\gamma(w)\partial_x\Psi_\gamma(\lambda, w,X_{\gamma}^\lambda)dw+\sqrt{\xi''(w)}d{W},
	%\end{split}
	\end{align}
	with the initial condition $X_{\gamma}^\lambda(s)=x.$
\end{proposition}
Alternatively, both Propositions \ref{cor1} and \ref{prop2} can also be verified by a direct application of It\^{o}'s formula. For a detailed proof, we refer the reader to \cite[Theorem 5]{C15}. Although the argument therein is for a different boundary condition and $\gamma(1-)$ is bounded, the same argument carries through with only minor modifications. Our next result gives a representation of the derivative of $\Psi_\gamma(\lambda,0,0)$ in $\lambda$ when $\lambda=0.$ 
\begin{proposition}
	\label{lem2}
	For any $0\leq q<1$ and any $\gamma\in\mathcal{U},$ the function $\lambda\mapsto\Psi_\gamma(\lambda,0,0)$ is continuously differentiable and its derivative at zero can be written as
	\begin{align}
	\label{lem2:eq1}
	\partial_\lambda\Psi_\gamma(0,0,0)=\e \partial_{x}\Phi_{\gamma}\bigl(q,X_{\gamma}^0(q)\bigr)^2,
	\end{align}
	where $X_{\gamma}^0$ is defined in \eqref{cor1:eq22} with $\lambda=0$, $s=0$ and $x=0.$
\end{proposition}
The rest of this subsection is devoted to verifying this proposition. Note that a direct computation gives $\partial_\lambda\Gamma_\gamma(0,1,\bx)= \Phi_\gamma(1,x_1)\cdot \partial_{x_2}\Phi_\gamma(1,x_2)$ for all $x_1,x_2\neq 0.$ From this, the validity of Proposition \ref{lem2} follows immediately from the Cole-Hopf transformation so long as $\gamma\in \mathcal{U}$ is a step function with finitely many jumps. To handle the general situation, we need the following three key lemmas, which relate the functionals in the Guerra-Talagrand bound and the Parisi formula as well as describe their regularity properties. They are taken directly from \cite{C15} and \cite{ChenHL}. The first one is in Lemma 6 in \cite{ChenHL}.
\begin{lemma}
	\label{lem1}
	Fix $0\leq q<1$ and $\gamma\in\mathcal{U}$. For any $w\in [q,1)$ and $\bx = (x_1,x_2) \in\mathbb{R}^2$,
	\begin{align}
	\begin{split}\label{sec4.3:lem1:eq1}
	\Gamma_\gamma(0,w,\bx)&=\Phi_{\gamma}(w,x_1)+\Phi_{\gamma}(w,x_2)
	\end{split}
	\end{align}
	and $\Gamma_\gamma(\lambda,w,\bx)$ is differentiable in $\lambda$ in the classical sense with
	\begin{align}
	\begin{split}\label{sec4.3:lem1:eq2}
	\partial_\lambda \Gamma_\gamma(0,w,\bx)&=\partial_x\Phi_{\gamma}(w,x_1)\partial_x\Phi_{\gamma}(w,x_2).
	\end{split}
	\end{align}
\end{lemma}
The second one can be found in Lemma 2 in \cite{C15}.
\begin{lemma}
	\label{lem-1}
	Let $\DD$ be a metric space and $I$ be an interval with right open edge. Let $f$ be a real-valued function on $\DD\times I$ and $f_0(\lambda)=\sup_{u\in \DD}f(u,\lambda)$ for $\lambda \in I$. Suppose that there exists a $\DD$-valued continuous function $u(\lambda)$ on $I$ such that $f_0(\lambda)=f(u(\lambda),\lambda)$ and $\partial_\lambda f$ is continuous on $\DD\times I$. Then $f_0(\lambda)$ is right-differentiable with the derivative given by $\partial_\lambda f(u(\lambda),\lambda)$ for all $\lambda\in I$.
\end{lemma}
We will also need the following Lemma 10 in \cite{ChenHL}.
\begin{lemma}\label{pde:lem}
	Let $0<r_0<r_1<r_2<\infty.$ Suppose that $\kappa_1,\kappa_2\in L^\infty([r_0 ,r_2]\times\mathbb{R})$ and $g\in L^\infty(\mathbb{R})$ with $\|\kappa_i\|_\infty\leq C_i$ for $i=1,2$ and $\|g\|_\infty\leq C_0$. Assume that $u$ is the classical solution to
	$$
	\partial_ru(r,x)=\partial_{xx}u(r,x)+\kappa_1(r,x)\partial_xu(r,x)+k_2(r,x),\,\,\forall (r,x)\in(r_0 ,{r_2} ]\times\mathbb{R}
	$$
	with the initial condition $u(r_0 ,x)=g(x).$ Then there exists a nonnegative continuous function $F$ on $[0,\infty)^3$ depending only on $r_0 ,r_2$ such that
	\begin{align*}
	\sup_{(r,x)\in[r_1 ,r_2]\times\mathbb{R}}|\partial_xu(r,x)|&\leq F(C_0,C_1,C_2).
	\end{align*}
\end{lemma}
The statement of this result in \cite{ChenHL} also assumed the continuity of $r\mapsto\|\partial_xu(r,\,\cdot\,)\|_\infty$. To avoid this,
we note that this assumption was used in \cite{ChenHL} only in the application of Gronwall's inequality, which also holds without it,
so this result in \cite{ChenHL} in fact did not need this assumption.

\medskip
{\noindent \bf Proof of Proposition \ref{lem2}.} 
First, we establish the differentiability of $\Psi_\gamma(\lambda,0,0)$. Again, recall that $\Psi_\gamma(\lambda, q,x)=\Gamma_\gamma(\lambda,q,x,x).$ If we denote\begin{align*}
F_\gamma^{q}(\lambda,u)&=\e\left[\Psi_\gamma\Bigl(\lambda,q,\int_0^q\xi''\gamma udw+\int_0^q\sqrt{\xi''}d{W}\Bigr)-\frac{1}{2}\int_0^q\xi''udw\right],
\end{align*}
for  $u\in D[q]$ then, by Proposition \ref{prop2},
\begin{align*}
\Psi_\gamma(\lambda,0,0)&=\max_{u\in D[q]} F_\gamma^{q}(\lambda,u).
\end{align*}
To compute the derivative of this in $\lambda$, we would like to apply Lemma \ref{lem-1} with $I=\Reals$, the metric space $\DD=(D[q],\|\,\cdot\,\|_q)$, and $u(\lambda)= u_\gamma^\lambda$ from Proposition \ref{prop2}. First of all, the proof of Lemma 6 in \cite{ChenHL} shows that $\Psi_\gamma(\lambda, q,x)=\Gamma_\gamma(\lambda,q,x,x)$ is differentiable in $\lambda$, the derivative $\partial_\lambda \Psi_\gamma(\lambda, q,x)$ is continuous in $(\lambda,x)$, and $|\partial_\lambda \Psi_\gamma(\lambda, q,x)|\leq 1.$ In particular, this implies that we can interchange the derivative and integral and write
\begin{align}
\partial_\lambda F_\gamma^q(\lambda,u)&=\e\left[\partial_\lambda\Psi_\gamma\Bigl(\lambda,q,\int_0^q\xi''\gamma u\,dw+\int_0^q\sqrt{\xi''}\,d{W}\Bigr)\right].
\label{pLfg}
\end{align}
For $u(\lambda)= u_\gamma^\lambda$ from Proposition \ref{prop2} with $x=0$ and $s=0$, 
$$
\int_0^q\xi''\gamma u_\gamma^\lambda\,dw+\int_0^q\sqrt{\xi''}\,d{W} = X_{\gamma}^{\lambda}(q),
$$
and, therefore, if we can show that $\partial_\lambda F_\gamma^q(\lambda,u)$ is continuous in $(\lambda, u)$ and $u_\gamma^\lambda$ is continuous in $\lambda$ on $\DD$ then Lemma \ref{lem-1} will imply that the right derivative
\begin{align}\label{lem2:proof:eq1}
\frac{d}{d\lambda}\Psi_\gamma(\lambda,0,0)
=
\partial_\lambda F_\gamma^q(\lambda,u_\gamma^\lambda)=\e\bigl[\partial_\lambda\Psi_\gamma\bigl(\lambda,q,X_{\gamma}^{\lambda}(q)\bigr)\bigr].
\end{align} 
Since $\partial_\lambda \Psi_\gamma$ is continuous in $(\lambda,x)$ and $|\partial_\lambda \Psi_\gamma|\leq 1,$ (\ref{pLfg}) implies that $\partial_\lambda F_\gamma^q$ is continuous in $(\lambda,u).$ 

It remains to show that $u_\gamma^\lambda$ is continuous in $\lambda$ on $\DD$. First, we note that $\Gamma_\gamma(\lambda,q,x,x)$ admits a similar expressions as Propositions \ref{cor1} and \ref{prop2}, which can be seen from Theorem 7 in \cite{ChenHL} with $(r,s)=(q,1)$ and $t=1$. This representation makes it clear how properties of the boundary condition $g(\lambda,\bx)$ in (\ref{eq1}) are propagated to $\Psi_\gamma(\lambda,q,x).$ For example, the fact that $g(\lambda,\bx)$ is $1$-Lipschitz in $\lambda$ and $2$-Lipschitz in $\bx$ implies that $|\partial_\lambda\Psi_\gamma(\lambda,q,x)|\leq 1$ and $|\partial_x\Psi_\gamma(\lambda,q,x)|\leq 2$. Using (\ref{prop2:eq2}), we get $|\partial_\lambda\Psi_\gamma(\lambda,s,x)|\leq 1$ and $|\partial_x\Psi_\gamma(\lambda,s,x)|\leq 2$ for all $s\in [0,q].$ From this, it can be shown that $\partial_x\Psi_\gamma(\lambda,s,x)$ is uniformly Lipschitz in $(\lambda,x)$ over all $s\in[0,q]$, as follows.

If $\gamma(ds)$ is an atomic measure with finitely many jumps then the solution $\Psi_\gamma$ of (\ref{pde3}) can be explicitly found via the Cole-Hopf transformation. Moreover, $A:=\partial_x\Psi_\gamma$ or $A:=\partial_\lambda\Psi_\gamma$ satisfies
\begin{align*}
\partial_t\Delta&=-\frac{\xi''}{2}\bigl(\partial_{xx}\Delta+2\gamma(\partial_x\Psi_\gamma)(\partial_x\Delta)\bigr)
\end{align*}
on $\mathbb{R}\times[0,q]\times\mathbb{R}$ with the boundary condition $A(\lambda,q,x)=\partial_x\Psi_\gamma(\lambda,q,x)$ or $\partial_\lambda\Psi_\gamma(\lambda,q,x)$. Since $|A(\lambda,q,x)|\leq 2$ and $|\partial_x\Psi_\gamma(\lambda,s,x)|\leq 2$ for $s\in[0,q]$, one may apply Lemma \ref{pde:lem} to obtain 
\begin{align}
\label{add:eq1}
\sup_{(\lambda,s,x)\in \mathbb{R}\times[0,q]\times \mathbb{R}}\bigl|\partial_x A(\lambda,s,x)\bigr|\leq F\bigl(\gamma(q)\bigr),
\end{align}
where $F$ is some nonnegative continuous function on $[0,\infty)$. Therefore, $\partial_x\Psi_\gamma(\lambda,s,x)$ is $F\bigl(\gamma(q)\bigr)$-Lipschitz in $\lambda$ and $x$ uniformly over $s\in [0,q]$, for atomic $\gamma(ds)$. For arbitrary $\gamma$ on $[0,q]$, we define the solution $\Psi_\gamma(\lambda,s,x)$ of (\ref{pde3}) with the boundary condition $\Gamma_\gamma(\lambda,q,x,x)$ via approximation of $\gamma$ by atomic $\gamma_n$ and ensuring, using the Arzela-Ascoli theorem and Cantor's diagonalization process, that $\Psi_{\gamma_n}(\lambda,s,x)$ and all its partial derivatives in $x$ and partial derivative in $\lambda$ converge to some $\Psi_{\gamma}(\lambda,s,x)$ and its partial derivatives. The  limiting $\Psi_\gamma(\lambda,s,x)$ is the solution of the equation (\ref{pde3}) with the boundary condition $\Gamma_\gamma(\lambda,q,x,x)$ and (\ref{add:eq1}) still holds, i.e. $\partial_x\Psi_\gamma(\lambda,s,x)$ is $F\bigl(\gamma(q)\bigr)$-Lipschitz in $\lambda$ and $x$ uniformly over $s\in [0,q]$. (The uniqueness of solutions follows from the same argument as in \cite{JagTob}.)

Finally, the equation \eqref{cor1:eq22} with $x=0$ and $s=0$ and (\ref{add:eq1}) imply that, for $0\leq w\leq q,$
$$
|X_{\gamma}^\lambda(w) - X_{\gamma}^{\lambda'}(w)|
\leq 
\xi''(q)\gamma(q)F(\gamma(q))\Bigl(|\lambda-\lambda'|w+ \int_0^w |X_{\gamma}^\lambda(t)-X_{\gamma}^{\lambda'}(t)|\,dt\Bigr).
$$
An application of Gronwall's inequality shows that $\lambda\mapsto X_\gamma^\lambda$ is continuous with respect to the sup-norm, and \eqref{add:eq1} implies that $\lambda \mapsto u_{\gamma}^\lambda(w)=\partial_x \Psi_\gamma(\lambda, w,X_{\gamma}^\lambda(w))$ is also continuous with respect to the sup-norm and, therefore, the norm $\|\,\cdot\,\|_q$ defined in \eqref{cor1:eq2}. As a result, $\lambda\mapsto(\lambda,u_\gamma^\lambda)$ is continuous and  (\ref{lem2:proof:eq1}) holds. Note that this right derivative is also continuous in $\lambda.$ It is well known (see e.g. \cite{B94}) that a function with continuous right derivative on an open interval  is continuously differentiable on this interval. Thus, $\Psi_\gamma(\lambda,0,0)$ is continuously differentiable. Finally, the verification of \eqref{lem2:eq1} follows directly from \eqref{lem2:proof:eq1} and \eqref{sec4.3:lem1:eq2}.
\qed

\subsection{Proof of Theorem \ref{Thm2}}

We now turn to the main proof of Theorem \ref{Thm2} by using the GT bound \eqref{RSB}. First of all, from the optimality of the Parisi measure $\gamma_P$, it is known that the points in the support of the Parisi measure must satisfy a system of consistency equations (see \cite[Proposition 3]{ChenHL}).

\begin{proposition}\label{prop1}
	For any $s\in \mathrm{supp}\,\gamma_P(ds)$,
	\begin{align}
	\begin{split}
	\label{prop1:eq1}
	\e\bigl(\partial_x\Phi_{\gamma_{P} }(s,X_{{\gamma_P} }(s))\bigr)^2&=s,
	\end{split}\\
	\begin{split}
	\label{prop1:eq2}
	\xi''(s)\e\bigl(\partial_{xx}\Phi_{\gamma_{P} }(s,X_{\gamma_{P} }(s))\bigr)^2&\leq 1,
	\end{split}
	\end{align}
	where $X_{\gamma_{P} }=(X_{\gamma_{P} }(w))_{0\leq w<1}$ is defined as in Proposition \ref{cor1} with $q=1,$ $s=0$ and $x=0.$
\end{proposition} 
A key consequence of these equations is the following lemma where the condition $K \geq 4$ is used.
\begin{lemma}
	\label{add:lem2}
	For even $K\geq 4,$ there exists $c\in (0,1)$ such that
	\begin{align*}
	\e\bigl(\partial_x\Phi_{\gamma_{P} }(s,X_{\gamma_P }(s))\bigr)^2<s
	\end{align*} 	
	for all $s\in (0,c).$
\end{lemma}
\textbf{\noindent Proof.}
Note that $\xi(s)=s^K$. By Proposition 2(ii) in \cite{ChenHL}, for any $0<q<1$, there exists a non-negative continuous function $F$ on $[0,\infty)$ depending only on $\xi$ and $q$ such that
\begin{align*}
\sup_{(s,x)\in [0,q]\times\mathbb{R}}\bigl|\partial_{xx}\Phi_{\gamma_P}(s,x)\bigr|\leq F(\gamma_P(q)).
\end{align*}
Since  $\xi''(0)=0$ for $K\geq 4$, this and \eqref{prop1:eq2} imply that there exists some $c\in (0,1)$ such that
\begin{align}\label{add:lem3:proof:eq1}
\xi''(s)\e\bigl(\partial_{xx}\Phi_{\gamma_{P} }(s,X_{\gamma_{P} }(s))\bigr)^2&< 1
\end{align}
for all $s\in(0,c).$ It remains to appeal to the following well known relationship between the left hand sides of the equations (\ref{prop1:eq1}) and (\ref{prop1:eq1}). By \cite[Lemma 3]{ChenHL}, 
\begin{align*}
&\partial_x\Phi_{\gamma_P}(s,X_\gamma(s))-\partial_x\Phi_{\gamma_P} (s',X_\gamma(s'))\\
&=\int_{s'}^s\xi''  (w)^{1/2}\partial_{xx}\Phi_{\gamma_P} (w,X_{\gamma_P}(w))dW(w)
\end{align*}
for all $0<s'<s<1$, and application of It\^{o}'s formula yields
\begin{align*}
&\e\bigl(\partial_{x}\Phi_{\gamma_{P} }(s,X_{\gamma_{P} }(s))\bigr)^2-\e\bigl(\partial_{x}\Phi_{\gamma_{P} }(s',X_{\gamma_{P} }(s'))\bigr)^2\\
&=\int_{s'}^s\xi''(w)\e\bigl(\partial_{xx}\Phi_{\gamma_P} (w,X_{\gamma_P}(w))\bigr)^2dw.
\end{align*}
In particular, letting $s'=0$ and using \eqref{add:lem3:proof:eq1}, 
\begin{align*}
\e\bigl(\partial_{x}\Phi_{\gamma_{P} }(s,X_{\gamma_{P} }(s))\bigr)^2&<s
\end{align*}
for any $s\in (0,c).$ 	This completes our proof.
\qed

\medskip
Let $c$ be the constant obtained in Lemma \ref{add:lem2}. For any fixed $q\in (0,c),$ we define 
$$
\gamma_q(s)=\frac{\gamma_P}{2}1_{[0,q)}+\gamma_P1_{[q,1)}\in \mathcal{U},
$$
where $1_A$ stands for the indicator function on the set $A.$ The definition of the boundary condition in (\ref{pde3}), $\Psi_\gamma(\lambda, q,x)=\Gamma_\gamma(\lambda,q,x,x)$, together with \eqref{sec4.3:lem1:eq1} yields 
$$
\Psi_{\gamma_q}(0,q,x)= \Gamma_{\gamma_q}(0,q,x,x) = 2\Phi_{\gamma_P}(q,x).
$$
Since $\gamma_q=\gamma_P/2$ on the interval $[0,q),$ comparing the equations (\ref{pde}) and \eqref{pde3} we see that
\begin{align}\label{add:eq4}
\Psi_{\gamma_q}(0,s,x)&=2\Phi_{\gamma_P}(s,x)
\,\,\mbox{ for }\,\, s\in[0,q].
\end{align}
Moreover,
\begin{align*}
\int_0^1s\xi''(s)\gamma_q(s)ds+\int_0^{q}s\xi''(s)\gamma_q(s)ds&=\int_0^1s\xi''(s)\gamma_P(s)ds,
\end{align*}
and, as a result,
\begin{align}
\label{add:eq3}
\mathcal{T}_q(0,\gamma_q)&=2\PP(\gamma_P)=2ME.
\end{align}
Next, recall $X_{\gamma_q}$ from \eqref{cor1:eq2} and $X_{\gamma_q}^0$ from \eqref{cor1:eq22} both with $x=0$ and $s=0$. Since on the interval $[0,q]$ we have $\gamma_q=\gamma_P/2$ and $\partial_x\Psi_{\gamma_q}(0,s,x)=2\partial_x\Phi_{\gamma_P}(s,x),$ \eqref{cor1:eq22} becomes
\begin{align*}
dX_{\gamma_q}^0&=\xi''\gamma_P\partial_x\Phi_{\gamma_q}(s,X_{\gamma_q}^0)ds+\sqrt{\xi''}dW.
\end{align*}
In other words, $X_{\gamma_q}$ and $X_{\gamma_q}^0$ are defined by the same equation on $[0,q]$ and, by the uniqueness of the strong SDE solution, we see that $X_{\gamma_q}^0=X_{\gamma_P}.$ Consequently, \eqref{lem2:eq1} implies
\begin{align*}
\partial_\lambda\Psi_{\gamma_q}(0,0,0)=\e \partial_x\Phi_{\gamma_P}(q,X_{\gamma_P}(q))^2
\end{align*}
and, by Lemma \ref{add:lem2},
\begin{align*}
\partial_\lambda \mathcal{T}_q(0,\gamma_q)
=
\partial_\lambda\Psi_{\gamma_q}(0,0,0) - q
=\e \partial_x\Phi_{\gamma_P}(q,X_{\gamma_P}(q))^2-q<0.
\end{align*}
This assures that, for all $q\in(0,c),$
\begin{align*}
\inf_{(\lambda,\gamma)\in \mathbb{R}\times\mathcal{U}}\mathcal{T}_q(\lambda,\gamma)<2ME.
\end{align*}
Since $\mathcal{T}_q(\lambda,\gamma)$ is continuous in $q$, the infimum on the left-hand side is an upper semi-continuous function of $q$ and, therefore, there exist $0<a<b<c$ and $\eta>0$ such that
\begin{align*}
\inf_{(\lambda,\gamma)\in \mathbb{R}\times\mathcal{U}}\mathcal{T}_q(\lambda,\gamma)<2ME-2\eta
\end{align*}
for all $q\in[a,b].$
Consequently, from \eqref{RSB} and the usual Gaussian concentration of measure, 
\begin{align*}
\e\MCE_N\bigl([a,b]\bigr)&\leq 2\e\ME_N-\eta
\end{align*}
for large enough $N$. Note that, since the Hamiltonian $H(\sigma)$ is symmetric for even $K$, i.e. $H(\sigma)=H(-\sigma)$, we see that this also implies that
\begin{align*}
\e\MCE_N\bigl([-b,-a]\cup[a,b]\bigr)&\leq 2\e\ME_N-\eta.
\end{align*}
This finishes our proof.
\qed

\section*{Acknowledgements} 
The authors thank a referee for careful reading and suggestions regarding the presentation of the paper.
D.G.~and M.R.~also thank organizers of the AIM workshop ``Phase transitions in randomized computational problems,"
where some of this work was conducted.


\begin{thebibliography}{99}
	
	\bibitem{ChenAufP} Auffinger, A., Chen, W.-K.:	The Parisi formula has a unique minimizer.
	\emph{Comm. Math. Phys.} \textbf{335} no. 3 (2015), pp. 1429--1444.
	
	\bibitem{ChenAuf} Auffinger, A., Chen, W.-K.: Parisi formula for the ground state energy in the mixed $p$-spin model.
	\emph{Ann. of Probab.} \textbf{45} no. 6B (2017), pp. 4617--4631.
	
	\bibitem{ACZ} Auffinger, A., Chen, W.-K., Zeng, Q.: The SK model is Full-step Replica Symmetry Breaking at zero temperature.
	\emph{preprint} (2017). \arXiv{1703.06872}.
	
	\bibitem{BV} Backhausz A., Vir\'{a}g B.: Spectral measures of factor of i.i.d. processes on vertex-transitive graphs.
	\emph{Ann. Henri Poincar\'{e}} (B) \textbf{53} no.~4 (2017), pp. 2260--2278.
	
	\bibitem{BS} Backhausz A., Szegedy B.: On the almost eigenvectors of random regular graphs,
	to appear in \emph{Ann. of Probab.} (2018). \arXiv{1607.04785}.
	
	\bibitem{BS2} Backhausz A., Szegedy B.: On large girth regular graphs and random processes on trees,
	to appear in \emph{Random Structures Algorithms} (2018), \arXiv{1406.4420}.
	
	\bibitem{BGT} Bayati, M., Gamarnik, D., Tetali P.: Combinatorial approach to the interpolation method
	and scaling limits in sparse random graphs. \emph{Ann. of Probab.} \textbf{41} (2013), pp. 4080--4115.
	
	\bibitem{BAJ} Ben Arous, G., Jagannath, A.: Spectral gap estimates in mean field spin glasses.
	\emph{Comm. Math. Phys.} \textbf{361} no.~1 (2018), pp.~1--52.
	
	\bibitem{B94} Bruckner, A.: Differentiation of real functions. CRM Monograph Series, Vol. 5, 2nd ed.,
	American Mathematical Society, Providence, (1994).
	
	\bibitem{C15} Chen, W.-K.: Variational representations for the Parisi functional and the two-dimensional Guerra-Talagrand bound.
	\emph{Ann. of Probab.} \textbf{45}, no. 6A (2017), pp. 3929--3966.
	
	\bibitem{ChenHL} Chen, W.-K., Handschy, M., Lerman, G.: On the energy landscape of the mixed even $p$-spin model.
	\emph{Probab. Theory Related Fields} \textbf{171} no.~1-2 (2018), pp. 53--95.
	
	\bibitem{ChenP17} Chen, W.-K., Panchenko, D.: Disorder chaos in some diluted spin glass models.
	\emph{Ann. of Applied Probab.} \textbf{28} no.~3 (2018), pp. 1356--1378.
	
	\bibitem{DMS} Dembo, A., Montanari, A., Sen, S.: Extremal cuts of sparse random graphs.
	\emph{Ann. of Probab.} \textbf{45} no.~2 (2017), pp. 1190--1217.
	
	\bibitem{DeSanctis} De Sanctis, L.: Random multi-overlap structures and cavity fields in diluted spin glasses. 
	\emph{J. Statist. Phys.} \textbf{117}  no. 5-6 (2004), pp. 785--799.
	
	\bibitem{EL} Elek, G., Lippner, G.: Borel oracles. An analytical approach to constant-time algorithms,
	\emph{Proc. Amer. Math. Soc} \textbf{138} (2010), pp. 2939--2947.
	
	\bibitem{FM} Fan, Z., Montanari, A.: How well do local algorithms solve semidefinite programs?
	\emph{STOC 2017}, ACM 604--614. \arXiv{1610.05350}.
	
	\bibitem{FS} Fleming, W. H., Soner, H. M.: Controlled Markov processes and viscosity solutions, volume
	25 of Stochastic Modelling and Applied Probability. Springer, New York, second edition (2006).
	
	\bibitem{FL} Franz, S., Leone, M.: Replica bounds for optimization problems and diluted spin systems.
	\emph{J. Statist. Phys.} \textbf{111} no. 3-4 (2003), pp. 535--564.
	
	\bibitem{GS} Gamarnik, D., Sudan, M.: Limits of local algorithms over sparse random graphs.
	\emph{Ann. of Probab.} \textbf{45} no.~4 (2017), pp. 2353--2376.
	
	\bibitem{GS2} Gamarnik, D., Sudan, M.: Performance of Sequential Local Algorithms for the Random NAE-K-SAT Problem, 
	\emph{SIAM J.~Comput.} \textbf{46} no.~2 (2017), pp. 590--619.
	
	\bibitem{Guerra} Guerra, F.: Broken replica symmetry bounds in the mean field spin glass model.
	\emph{Comm. Math. Phys.} \textbf{233} no.~1 (2003), pp. 1--12.
	
	\bibitem{GT2} Guerra, F., Toninelli, F.~L.: The thermodynamic limit in mean field spin glass models.
	\emph{Comm. Math. Phys.} \textbf{230} no.~1 (2001), pp. 71--79.
	
	\bibitem{GT} Guerra, F., Toninelli, F.L.: The high temperature region of the Viana-Bray diluted spin glass model.
	\emph{J. Statist. Phys.} \textbf{115} no.~1-2 (2004), pp. 531--555.
	
	\bibitem{HV} Harangi, V., Vir\'{a}g B.: Independence ratio and random eigenvectors in transitive graphs.
	\emph{Ann. of Probab.} \textbf{43} no.~5 (2015), pp. 2810--2840.
	
	\bibitem{HLS} Hamati H., Lovasz L., Szegedy B.: Limits of local-global convergent graph sequences.
	\emph{Geom.~Funct.~Anal.}~\textbf{24} no.~1 (2014), pp. 269--296.
	
	\bibitem{HW} Hoppen C., Wormald N.: Local algorithms, regular graphs of large girth, and random regular graphs.
	to appear in \emph{Combinatorica} (2018). \arXiv{1308.0266}.
	
	\bibitem{JagTob} Jagannath, A., Tobasco, I.: A dynamic programming approach to the Parisi functional.
	\emph{Proc. Amer. Math. Soc.} \textbf{144} (2016), pp. 3135--3150.
	
	\bibitem{JKS} Jagannath, A., Ko, J., Sen, S.: A connection between MAX $\kappa$-CUT and the
	inhomogeneous Potts spin glass in the large degree limit.
	\emph{Ann. of Applied Probab.} \textbf{28} no.~3 (2018), pp. 1536--1572.
	
	\bibitem{KS} Karatzas, I., Shreve, S.~E.: Brownian motion and stochastic calculus.
	Volume 113 of Graduate Texts in Mathematics. Springer-Verlag, New York, 2nd ed (1991).
	
	\bibitem{Lyons} Lyons, R.: Factor of IID on trees,
	\emph{Combin. Probab. Comput.} \textbf{26} no.~2 (2015), pp. 285--300.
	
	\bibitem{LN} Lyons, R., Nazarov, F.: Perfect matchings as iid factors on non-amenable groups.
	\emph{European J. Combin.}~\textbf{32} (2011), pp. 1115--1125.
	
	\bibitem{PT} Panchenko, D., Talagrand, M.: Bounds for diluted mean-fields spin glass models.
	\emph{Probab. Theory Related Fields} \textbf{130} no.~3 (2004), pp. 319--336.
	
	\bibitem{SKmodel} Panchenko, D.: The Sherrington-Kirkpatrick Model.
	Springer Monographs in Mathematics. Springer-Verlag, New York (2013).
	
	\bibitem{PUltra} Panchenko, D.: The Parisi ultrametricity conjecture.
	\emph{Ann. of Math.} (2) \textbf{177} no.~1 (2013), pp. 383--393.
	
	\bibitem{PPF} Panchenko, D.: The Parisi formula for mixed $p$-spin models.
	\emph{Ann. of Probab.} \textbf{42} no.~3 (2014), pp. 946--958.
	
	\bibitem{PKsat} Panchenko, D.: On the $K$-sat model with large number of clauses.
	\emph{Random Structures Algorithms} \textbf{52} no.~3 (2018), pp. 536–542
	
	\bibitem{Parisi79} Parisi, G.: Infinite number of order parameters for spin-glasses.
	\emph{Phys. Rev. Lett.} \textbf{43} (1979), pp. 1754--1756.
	
	\bibitem{Parisi} Parisi, G.: A sequence of approximate solutions to the S-K model for spin glasses.
	\emph{J. Phys. A} \textbf{13} L-115 (1980).
	
	\bibitem{R} Rahman, M: Factor of IID percolation on trees,
	\emph{SIAM J. Discrete Math.} \textbf{30} no.~4 (2016), pp. 2217--2242.
	
	\bibitem{RV} Rahman, M., Vir\'{a}g B.: Local algorithms for independent sets are half--optimal.
	\emph{Ann. of Probab.} \textbf{45} no.~3 (2017), pp. 1543--1577.
	
	\bibitem{Sen} Sen, S.: Optimization on sparse random hypergraphs and spin glasses.
	to appear in \emph{Random Structures Algorithms}. \arXiv{1606.02365}.
	
	\bibitem{TalUltra} Talagrand, M.: Mean field models for spin glasses: some obnoxious problems.
	Lecture Notes in Mathematics, Vol. 1900 (2007), pp. 63--80.
	
	\bibitem{TPF} Talagrand, M.: The Parisi formula.
	\emph{Ann. of Math.} (2) \textbf{163} no.~1 (2006), pp. 221--263.
	
	\bibitem{SG2} Talagrand, M.: Mean-Field Models for Spin Glasses. Ergebnisse der Mathematik und ihrer Grenzgebiete. 3.
	Folge A Series of Modern Surveys in Mathematics, Vol. 54, 55. Springer-Verlag (2011).
	
\end{thebibliography}
\end{document}